\documentclass[article,leqno]{article}
\usepackage[a4paper, total={7in, 9in}]{geometry}
\usepackage[english]{babel}
\usepackage{tgpagella}
\usepackage{graphicx}
\usepackage{xcolor}
\usepackage{float}% Required for inserting images
\usepackage{amsmath}
\usepackage{stmaryrd}
\usepackage{parskip}
\usepackage{capt-of}
\usepackage{amsfonts}
\usepackage{latexsym}
\usepackage{mathrsfs}
\usepackage{csquotes}
\usepackage{etoolbox}
\usepackage{hyperref}
\usepackage{enumitem}

\hypersetup{
  colorlinks=true,
  linkcolor=black,
  citecolor=teal
}
\usepackage{xcolor}
\usepackage{tikz-cd}
\usetikzlibrary{calc}
\tikzcdset{scale cd/.style={every label/.append style={scale=#1},
    cells={nodes={scale=#1}}}}
\usepackage{amsthm}
\usepackage{accents}
\usepackage{mathtools}
\usepackage{authblk}
\newcommand{\ZZ}{\mathbb{Z}}
\newcommand{\CC}{\mathbb{C}}

\newcommand{\QQ}{\mathbb{Q}}

\newcommand{\g}{\mathfrak{g}}

\newcommand{\sh}{\textnormal{sh}}

\newtheoremstyle{mytheoremstyle}%
{}%space above theorem
{}%space below theorem
{\itshape}%body font
{}%indentation amount
{\bfseries}%theorem head font
{.}%punctuation separator
{ }%space after theorem head
{\thmname{#1}\thmnumber{ #2}\thmnote{ \textnormal{#3}}}
\theoremstyle{mytheoremstyle}
\newtheorem{theorem}{Theorem}[section]
\newtheorem*{theorem*}{Theorem}
\newtheorem{corollary}[theorem]{Corollary}

\newtheorem{proposition}[theorem]{Proposition}
\newtheorem{lemma}[theorem]{Lemma}
\newtheorem*{lemma*}{Lemma}

\theoremstyle{definition}
\newtheorem{definition}[theorem]{Definition}
\newtheorem{example}[theorem]{Example}
\AtEndEnvironment{example}{\hfill$\diamondsuit$}

\theoremstyle{remark}
\newtheorem{remark}[theorem]{Remark}

\usepackage[colorinlistoftodos]{todonotes}

\usepackage[backend=biber,style=alphabetic,maxbibnames=9,maxalphanames=9, minalphanames=4]{biblatex}

\title{Braid group invariance of quantum tori and applications}
\author{Francesca Paganelli}

\date{}

\begin{document}

\maketitle
\begin{abstract}
We establish that the skew-symmetric bicharacter
on the group of $\ell$-weights of a quantum affine algebra $\mathcal{U}_q(\hat{\g})$ is invariant under Chari's braid group action. Furthermore, we define an extension of this bicharacter to the group generated by prefundamental $\ell$-weights and show that this invariance continues to hold under the extended braid group action of Frenkel--Hernandez. As immediate applications, we obtain braid group actions by automorphisms on the quantum tori associated respectively with the category of finite-dimensional representations of $\mathcal{U}_q(\hat{\g})$ and with the category $\mathcal{O}^{\mathfrak{b}}$ of its Borel subalgebra. In addition, we derive explicit formulas for quantum $QQ$-systems in all simply-laced types. Finally, we show that our extended bicharacter yields a quantization matrix compatible with the Hernandez--Leclerc cluster algebra categorified by $\mathcal{O}^{\mathfrak{b}}$ for all finite Dynkin types, extending Bittmann's construction in the simply-laced setting.
    
\end{abstract}
\tableofcontents
\section{Introduction}
\hypertarget{target:intro}{}
A well known fact in classical Lie theory is that the Weyl group $W$ associated with a simple, complex, finite-dimensional Lie algebra $\mathfrak{g}$ acts by isometries on $\mathfrak{h}^*$, the dual of the Cartan subalgebra. In other words, the symmetric bilinear form defined on $\mathfrak{h}^*$ (and in particular on the weight lattice) is invariant under the $W$-action. In this paper we prove a quantum affine analogue of the invariance phenomenon for a skew-symmetric bicharacter. Namely, let $I=\{1,...,n\}$ be the set of vertices of the Dynkin diagram of $\mathfrak{g}$ and let $q$ be a non-zero complex number, not root of unity. We denote by $\mathcal{M}$ the multiplicative group generated by the elements $\{Y_{i,q^r}\}_{i\in I,r\in\ZZ}$, which can be thought as affine analogs of the fundamental weights of $\mathfrak{g}$. By results of Chari \cite{chari}, the braid group $\mathcal{B}$ associated with $\mathfrak{g}$ acts on $\mathcal{M}$ by automorphisms. Then, we consider $\Omega$, the skew-symmetric bicharacter on $\mathcal{M}$ defined in \cite{nak04, vv, halgebraicqt}. Our main result, that we prove in greater generality in Theorem \ref{main thm}, is the following:
\begin{theorem*}\label{intro main thm}
    The bicharacter $\Omega$ is $\mathcal{B}$-invariant: for every $T\in\mathcal{B}$ acting on $\mathcal{M}$, and for every $M,N\in \mathcal{M}$ we have
    \[\Omega(T(M),T(N))=\Omega(M,N).\]
\end{theorem*}
We now provide further details regarding the motivation and context for this result.

\textit{Context.} For $\mathfrak{g}$ and $q$ as above, the quantum affine algebra $\mathcal{U}_q(\hat{\mathfrak{g}})$ was introduced by Drinfeld in the 80's as a deformation of the universal enveloping algebra of the affine Kac--Moody algebra $\hat{\mathfrak{g}}$ (in general $q$ could be a root of $1$, but we keep assuming that it is not the case). We refer to the ICM lecture \cite{drinfeld} for an introduction. The representation theory of these algebras is very rich and plays a significant role in the study of spectra of quantum integrable models (see for example the review in \cite{her_bourbaki} and references in \cite{fr}). In particular, the following categories have been extensively studied: the category $\mathcal{C}$ of finite dimensional $\mathcal{U}_q(\hat{\mathfrak{g}})$-modules and the category $\mathcal{O}^{\mathfrak{b}}$ of representations of the Borel subalgebra $\mathcal{U}_q(\hat{\mathfrak{b}})$ of $\mathcal{U}_q(\hat{\mathfrak{g}})$ (see \cite{hj12}). Actually, one can restrict the attention to the integer skeleton subcategories $\mathcal{C}_{\ZZ}\subset{\mathcal{\mathcal{C}}}$ and $\mathcal{O}^{\mathfrak{b}}_{\ZZ}\subset\mathcal{O}$; hence we present our results in this setting. Note that every simple in $\mathcal{C}_{\ZZ}$ is still simple as a $\mathcal{U}_q(\hat{\mathfrak{b}})$-module and $\mathcal{C}_{\ZZ}$ can be seen as a subcategory of $\mathcal{O}^{\mathfrak{b}}_{\ZZ}$.

Representations in the categories $\mathcal{C}_{\ZZ}$ and $\mathcal{O}^{\mathfrak{b}}_{\ZZ}$ decompose into direct sums of $\ell$-weight spaces, quantum-affine analogs of classical weight-spaces (\cite{chari94},\cite{hj12}). This decomposition is encoded in the $q$-\textit{character} of a representation. Namely, Frenkel--Reshetikhin introduced in \cite{fr} an injective ring morphism, called the $q$-character morphism, that assigns to every class in the Grothendieck ring $K(\mathcal{C}_{\ZZ})$ a Laurent polynomial in $\mathcal{Y}:=\ZZ\left[Y_{i,q^r}^{\pm 1}\mid i\in I,r\in \ZZ\right]$:
 \[\chi_q:K(\mathcal{C}_{\ZZ})\to \mathcal{Y}.\] Note that thanks to the injectivity of $\chi_q$, $K(\mathcal{C})$ is a commutative ring. The simple representations with highest $\ell$-weight $Y_{i,q^r}$, $i\in I,a\in\CC^*$, are called the fundamental representations. They are denoted by $L(Y_{i,q^r})$. To define the $q$-character of representations in $\mathcal{O}^{\mathfrak{b}}_{\ZZ}$ one has to extend the ring $\mathcal{Y}$, so that more relations in the Grothendieck ring can be captured, for example the Baxter relation. To do so, we can introduce the ring $\mathcal{Y}':=\ZZ\left[\Psi_{i,q^r}^{\pm 1}\mid i\in I,a\in\ZZ\right]$ with the relation:
 \[Y_{i,q^r}=\Psi_{i,q^{r-r_i}}\Psi_{i,q^{r+r_i}}^{-1}, \] for every $i\in I$ and $a\in\ZZ$, so that we have $\mathcal{Y}\subset\mathcal{Y}'$. Note that we use here a re-normalized version such as in \cite{ghl}, so that weight factors can be omitted.
  Now, as the objects in $\mathcal{O}^{\mathfrak{b}}_{\ZZ}$ can be infinite-dimensional, the $q$-character of a class in the Grothendieck ring $K(\mathcal{O}^{\mathfrak{b}}_{\ZZ})$ is an element of $\widetilde{\mathcal{Y}_e'}$, the completion by formal power series of $\mathcal{Y}'$. The simple representations with highest $\ell$-weight $\Psi_{i,q^r}$ (resp. $\Psi_{i,q^r}^{-1}$) are called the positive (resp. negative) prefundamental representations, they are infinite-dimensional and are denoted by $L(\Psi_{i,q^r})$ (resp. $L(\Psi_{i,q^r}^{- 1})$). 

 As mentioned above, Chari in \cite{chari} defined an action by automorphisms of $\mathcal{B}$ on $\mathcal{Y}$. Later, her definition was extended by Frenkel--Hernandez in \cite{fh23} on the larger ring $\mathcal{Y}'$. They defined automorphisms $T_i'$ of $\mathcal{Y}'$ such that:
 \begin{align*}
       T_i'(\Psi_{j,q^r})&:=\Psi_{j,q^r}\quad\textit{if}\ i\neq j,\\
       T_i'(\Psi_{i,q^r})&:=\Psi_{i,q^{r-2r_i}}^{-1}\prod_{\{k\in I\mid C_{ik}=-1\}}\Psi_{k,q^{r-r_i}}\prod_{\{k\in I\mid C_{ik}=-2 \}}\Psi_{k,q^r}\Psi_{k,q^{r-2}}\prod_{\{k\in I\mid C_{ik}=-3\}}\Psi_{k,q^{r-1}}\Psi_{k,q^{r+1}}\Psi_{k,q^{r-3}}.
    \end{align*}
 This definition is a crucial point in Frenkel--Hernandez's construction of a Weyl group action on a direct sum of completions of $\mathcal{Y}'$ (see \cite{fh22},\cite{fh23}). 

 Finally, we have to introduce the bicharacter $\Omega$. It was defined by Hernandez \cite{halgebraicqt} (with a different notation) in the algebraic construction of the quantum Grothendieck ring $K_t(\mathcal{C}_{\ZZ})$ of the category $\mathcal{C}_{\ZZ}$. The ring $K_t(\mathcal{C})$ is a non-commutative deformation, via a formal variable $t$, of $K(\mathcal{C})$, first introduced geometrically in simply-laced types by Nakajima \cite{nak04} and Varagnolo--Vasserot \cite{vv} using quiver varieties. Quantum Grothendieck rings are particularly important to deal with Kazhdan-Lusztig type algorithms to compute $q$-characters. The bicharacter $\Omega$ is used to define the quantum torus $(\mathcal{Y}_t,*)$, so that it contains $K_t(\mathcal{C}_{\ZZ})$. Specifically, $\mathcal{Y}_t$ is the $\ZZ[t^{\pm \frac{1}{2}}]$-algebra generated by $\{Y_{i,q^r}^{\pm 1}\}_{i\in I,r\in\ZZ}$ subject to the relations
 \[Y_{i,q^r}*Y_{j,q^s}=t^{\Omega(Y_{i,q^r},Y_{j,q^s})}Y_{j,q^s}*Y_{i,q^r},\] for every $i,j\in I$ and $r,s\in \ZZ$. We call this type of relations the $t$-\textit{relations}. The definition of $\Omega$ is given in terms of \textit{inverse quantum Cartan matrix} $\tilde{C}(z)$, where $z$ is an indeterminate (Definition \ref{def quantum cartan}). For every $i,j\in I$, the Laurent expansion at $z=0$ of the entry $\tilde{C}_{ij}(z)$ is
 \[\tilde{C}_{ij}(z)=\sum_{u\in\ZZ}\tilde{c}_{ij}(u)z^u\in \ZZ((z)).\]  Then, for every $i,j\in I$ and $r,s\in \ZZ$, Hernandez defined
 \[\Omega(Y_{i,q^r},Y_{j,q^s})=\tilde{c}_{ij}(m+r_i)-\tilde{c}_{ij}(m-r_i)+\tilde{c}_{ij}(-m-r_i)-\tilde{c}_{ij}(-m+r_i),\]
 where $m=s-r$ and $r_i$'s are the diagonal entries of the symmetrizer of the Cartan matrix of $\g$. 
 Note that the bicharacter $\Omega$ is defined for every finite Dynkin type. 

 Later, Bittmann \cite{b21} extended the definition of this bicharacter to the larger ring $\mathcal{Y}'$, though only for $\mathfrak{g}$ of simply-laced type. In this way she defined the quantum torus $\mathcal{Y}'_t$ (in analogy with $\mathcal{Y}_t$). Her motivation was the construction of a quantum Grothendieck ring for the category $\mathcal{O}^{\mathfrak{b}}_{\mathbb{Z}}$, so that $K_t(\mathcal{O}^{\mathfrak{b}}_{\mathbb{Z}})$ defined in \cite{b21} lives inside a completion of $\mathcal{Y}'_t$.

 \textit{New results.} This paper begins with a review of the coefficients of the Laurent expansion of the inverse quantum Cartan matrix. We establish a system of relations between the coefficients $\tilde{c}_{ij}(u)$ (Lemma \ref{relation quantum cartan}), which is used several times throughout the paper.
% \begin{lemma*}[(Lem. \ref{relation quantum cartan})]
%     If ($i=j$ and $u\neq 0$) or ($i\neq j$ and $u\in\ZZ$) we have
%            \begin{gather*}
%                \tilde{c}_{ij}(u-r_i)+\tilde{c}_{ij}(u+r_i)-\sum_{\{k\in I\mid c_{ik=-1}\}}\tilde{c}_{kj}(u)-\sum_{\{k\in I\mid c_{ik=-2}\}}\Big(\tilde{c}_{kj}(u+1)+\tilde{c}_{kj}(u-1)\Big)-\\-\sum_{\{k\in I\mid c_{ik=-3}\}}\Big(\tilde{c}_{kj}(u-2)+\tilde{c}_{kj}(u)+\tilde{c}_{kj}(u+2)\Big)=0.
%            \end{gather*}
% \end{lemma*}
This generalizes Lemma 3.2.4 in \cite{b21} to all finite Dynkin types. To present our results in the greatest possible generality, we also extend the definition of the bicharacter on $\mathcal{Y}'$ to all finite Dynkin types. Namely, for every underlying Dynkin type we define a skew-symmetric map $\Omega'$ (Definition \ref{def F}) on $\mathcal{Y}'\times\mathcal{Y}'$, that coincides with the bicharacter in \cite{b21} in the simply-laced case and restricts to the one in \cite{halgebraicqt} on $\mathcal{Y}\times\mathcal{Y}$ (see Lemma \ref{restrizione Lambda}). Having these definitions, using Lemma \ref{relation quantum cartan} we prove Theorem \ref{intro main thm} and the more general version Theorem \ref{main thm}, establishing the invariance of $\Omega'$ under the action of the braid group on $\mathcal{Y}'$ (this result is also new in simply-laced types). The rest of the paper is dedicated to some applications of our definition of $\Omega'$ and of the invariance Theorem. 
 
As a first direct application, using the bicharacter $\Omega'$, we define the quantum torus $(\mathcal{Y}'_t,\star)$, which naturally generalizes the construction in \cite{b21} to all Dynkin types. In particular, the quantum torus $\mathcal{Y}_t$ is a subring of $\mathcal{Y}_t'$ (Proposition \ref{inclusion quantum tori}). Then, the braid group-invariance theorem (Theorem \ref{main thm}) implies the existence of a braid group action by automorphisms on $\mathcal{Y}_t'$ (Theorem \ref{teo q tori}). This was in fact one of the motivations for this work: in \cite{fh22}, to construct a Weyl group action on a direct sum of completions of $\mathcal{Y}$, the authors defined automorphisms $\Theta_i$, $i\in I$, such that the leading term of $\Theta_i(Y_{i,q^r})$ is governed by the braid group operator $T_i$. Our new understanding of the braid group action on the quantum tori $\mathcal{Y}'_t$ and $\mathcal{Y}_t$ is the first step toward quantizing the Weyl group action in \cite{fh23} and \cite{fh22}. We plan to discuss this in another paper. Note that it is a desirable result since it would reveal new symmetries on the $(q,t)$-\textit{characters}, defined geometrically by Nakajima \cite{nak04} and algebraically by Hernandez \cite{halgebraicqt} (and completely available at the moment only for modules in $\mathcal{C}_{\ZZ}$).

We then focus on other applications. We begin Section \ref{section QQ} by recalling results from \cite{fh23} and \cite{ghl} concerning cluster algebras and $QQ$-systems, the latter being important relations in $K(\mathcal{O}^{\sh}_{\ZZ})$, the Grothendieck ring of the category $\mathcal{O}^{\sh}_{\ZZ}$ of representations of shifted quantum affine algebras (we refer the reader to the foundational papers \cite{ft}, \cite{sqaahernandez} for definitions and properties on these algebras). Indeed, when $\mathfrak{g}$ is simply-laced, in \cite{ghl} the authors constructed a cluster algebra $\mathcal{A}$ such that $K(\mathcal{O}^{\sh}_{\ZZ})$ is isomorphic to a completion of $\mathcal{A}$. Under this correspondence, $QQ$-systems arise as a subset of the exchange relations in $\mathcal{A}$. It is explained in \cite{ghl} that $K(\mathcal{O}^{\mathfrak{b}}_{\ZZ})$ is a subring of $\mathcal{O}^{\sh}$. In addition, thanks to the $q$-character morphism (defined also in this setting), $K(\mathcal{O}^{\sh}_{\ZZ})$ is isomorphic to $\widetilde{\mathcal{Y}_e'}$, the completion by formal power series of $\mathcal{Y}'$. However, the variables appearing in the $QQ$-systems, called $\underline{Q}$-variables, have leading term determined via the braid group action on $\mathcal{Y}'$. In \cite{pag25} we constructed a quantization $\mathcal{A}_t$ of $\mathcal{A}$ (in the sense of quantum cluster algebra). In particular, $\mathcal{A}_t$ is a subring of the quantum torus $(\mathcal{Y}'_t,\star)$. We showed that quantum $QQ$-system can be obtained as quantum exchange relations in $\mathcal{A}_t$, though we could write them explicitly only in type $A_1$. By applying Theorem \ref{main thm}  and some key properties about $\underline{Q}$-variables we are now able to state precisely the quantum $QQ$-systems in general:
\begin{theorem*}[(Thm. \ref{preciseqQQ})]
    For $(i,r)\in V$ and $w\in W$ such that $ws_i>w$, the following identity holds in $\mathcal{A}_t$:
    \begin{equation}
        \underline{Q}_{w(\omega_i),q^{r-2}}\star\underline{Q}_{ws_i(\omega_i),q^r}= t^A\underline{Q}_{ws_i(\omega_i),q^{r-2}}\star\underline{Q}_{w(\omega_i),q^r} +\prod_{\{j\in I\mid C_{ij}=-1\}}\underline{Q}_{w(\omega_j),q^{r-1}},
    \end{equation}
    where 
    \[A=-\frac{3}{2}-\frac{\tilde{c}_{ii}(3)}{2}+\sum_{\{k\in I\mid C_{ik}=-1\}}\frac{\tilde{c}_{ki}(2)}{2},\] and $s_i$ denotes a simple reflection in $W$.
\end{theorem*}

Finally, we apply our generalized definition of the bicharacter $\Omega'$ in the context of monoidal categorification of cluster algebras. This branch of research has been extensively explored, we recall some of the main contributions \cite{hlDuke}, \cite{hl16o}, \cite{kkopI}, \cite{kkopII}, \cite{qin17}. The category $\mathcal{O}^{\mathfrak{b}}_{\ZZ}$ has two important subcategories, denoted by $\mathcal{O}^{\mathfrak{b},+}_{\ZZ}$ and $\mathcal{O}^{\mathfrak{b},-}_{\ZZ}$, obtained imposing some conditions on the $\ell$-weights. Both $\mathcal{O}^{\mathfrak{b},\pm}_{\ZZ}$ have $\mathcal{C}_{\ZZ}$ as subcategory. In \cite{hl16o} the authors proved that for all Dynkin types $K(\mathcal{O}^{\mathfrak{b},\pm }_{\ZZ})$, the Grothendieck rings of these categories, admit a cluster algebra structure, providing an infinite quiver $\Gamma$, with vertex set $V'\subset I\times \ZZ$, which serves as initial seed for both categories. Let us denote $\mathcal{A}'$ the cluster algebra (to not be confused with $\mathcal{A}$). Bittmann \cite{b21} provided, for simply laced-types, a quantization of $\mathcal{A}'$, using $\mathcal{Y}_t$ as ambient quantum torus. She then defined the quantum Grothendieck ring $K_t(\mathcal{O}^{\mathfrak{b},+}_{\ZZ})$ as a completion of the quantum cluster algebra and proved that $K_t(\mathcal{C}_{\ZZ})$ is a subring of $K_t(\mathcal{O}^{\mathfrak{b},+}_{\ZZ})$.  In Section \ref{sec q cluster} we complete this program, defining a quantization of $\mathcal{A}$ for all Dynkin types. Namely, we define a skew-symmetric integer matrix $\Lambda$ using the skew-symmetric bicharacter $\Omega'$ and we prove in Theorem \ref{thm compatibility} that $\Lambda$ is compatible with the quiver $\Gamma$. By definition this means that the product of the exchange matrix $B$ of $\Gamma$ with $\Lambda$ is a non-zero multiple of the identity matrix $\mathrm{Id}_{V'}$. In \cite{pag25}, a fundamental ingredient to obtain $\mathcal{A}_t$ is the quantization of $\mathcal{A}'$ provided in \cite{b21}, but we worked only in simply-laced type as the results in \cite{ghl} were (and still are) available only in this setting. Therefore, Theorem \ref{thm compatibility} indicates that once the construction in \cite{ghl} is extended to all types, the corresponding quantization will be both possible and natural.

\textit{Organization.} In Section \ref{sec quantum cartan} we recall the definition of the quantum Cartan matrix, of its inverse, and we review some of their properties. We also prove Lemma \ref{relation quantum cartan}, which is crucial throughout the paper. Then we present the bicharacters from \cite{halgebraicqt} and \cite{b21} and we give our general definition of $\Omega'$. In Section \ref{sec braid} we review the braid group actions from \cite{chari} and \cite{fh23} and we prove our main theorem (Theorem \ref{main thm}) about the invariance of $\Omega'$ under these actions. In Section \ref{sec q tori} we consider the quantum torus $(\mathcal{Y}'_t,\star)$ and prove that it admits a braid group action. Section \ref{section QQ} is dedicated to quantum $QQ$-systems, that we establish completely for every simply-laced type. In Section \ref{sec q cluster} we construct a skew-symmetric matrix $\Lambda$ sing $\Omega'$ and we prove that it is compatible with the cluster algebra categorified by the category $\mathcal{O}^{\mathfrak{b},+}_{\ZZ}$, extending to all Dynkin types the construction done in \cite{b21} in simply-laced setting. 
\subsection*{Acknowledgments} 
I would like to thank my PhD supervisors, David Hernandez and Giovanni Cerulli Irelli, for their constant guidance and encouragements. 
\section{Quantum Cartan matrices and bicharacters}\label{section q cartan}
In this section we recall the notion of quantum Cartan matrix $C(z)$ and its inverse $\tilde{C}(z)$, which is a square matrix with entries in $\QQ(z)$. In particular, we can consider the Laurent expansion at $z=0$: $\tilde{C}_{ij}(z)=\sum_{u\in\ZZ}\tilde{c}_{ij}(u)z^u$. We prove that the coefficients $\tilde{c}_{ij}(u)$  satisfy a key relation for all finite Dynkin type. We will use such relation several times throughout the paper. Then we recall some rings of Laurent polynomials in infinite variables, $\mathcal{Y}\subset \mathcal{Y}'$ and we define a skew-symmetric bicharacter on the multiplicative group of monomials in $\mathcal{Y}'$. Our definition extends to all finite Dynkin types the definition given by Bittmann \cite{b21} and is compatible with the original definition in \cite{halgebraicqt} (given for $\mathcal{Y}$) . 
\subsection*{Quantum Cartan matrices}\label{sec quantum cartan}
Let $\mathfrak{g}$ be a finite-dimensional, complex, simple Lie algebra. The classification of these algebras is given by finite-type Dynkin diagrams and we follow the conventions from \cite{kac}. Let $n$ be the rank of $\mathfrak{g}$ and $I=\{1,...,n\}$. Let $C=(C_{ij})_{i,j\in I}$ be the Cartan matrix of $\mathfrak{g}$. Let $D=\mathrm{diag}(r_i\mid i\in I)$ be the unique matrix such that $DC$ is symmetric and the $r_i$'s are positive, relatively prime integers. When $\g$ is of simply-laced type, then $C$ is symmetric and so in this case $D=\mathrm{I}_n$, the identity matrix. Let $z$ be an indeterminate. For every integer $k$, we use the standard convention
\begin{equation}
    [k]_z=\frac{z^{k}-z^{-k}}{z-z^{-1}}.
\end{equation}
\begin{remark}\label{remark sugli r_i}
    The condition on $DC$ to be symmetric means
    \[C_{kj}r_k=C_{jk}r_j,\quad \forall k,j\in I.\]
    Thus, if $C_{kj}<-1$ (namely $C_{kj}=-2$ or $C_{kj}=-3$), then by the structure of Cartan matrices we have $C_{jk}=-1$. Thus $r_k=1$, by the coprimality of the $(r_i)_{i\in I}$.
\end{remark}
\begin{definition}\label{def quantum cartan}
     The \textit{quantum Cartan matrix} associated with $C$ is the $I\times I$-matrix $C(z)=(C_{ij}(z))_{i,j\in I}$ with entries in $\ZZ[z^{\pm 1}]$ given by
     \[C_{ij}(z)=\begin{cases}
         z^{r_i}+z^{-r_i}\quad &\text{if}\ i=j,\\
         [C_{ij}]_z\quad &\text{if}\ i\neq j.
     \end{cases}\]
\end{definition}
Since the Cartan matrix $C(1)$ is invertible, $C(z)$ is invertible too and we denote by $\tilde{C}(z)$ its inverse, whose entries are rational functions in $\QQ(z)$. For each entry of the inverse quantum Cartan matrix we can consider its Laurent expansion at $z=0$:
\[\tilde{C}_{ij}(z)=\sum_{u\in\ZZ}\tilde{c}_{ij}(u)z^u\in \QQ((z)).\] 
\begin{proposition}[{\cite[Lemma 3.3]{fo_qdata}}]\label{prop tecnica}
    For any $i,j$ in $I$ and $u$ in $\ZZ$, we have that $\tilde{c}_{ij}(u)\in\ZZ$. Moreover:
    \begin{enumerate}
        \item $\tilde{c}_{ij}(u)=0$ if $u<r_i$,\label{1}
        \item $\tilde{c}_{ij}(r_i)=\delta_{ij}$.\label{2}
    \end{enumerate}
\end{proposition}

The next result will be fundamental for the rest of the paper.
\begin{lemma}\label{relation quantum cartan}
    If ($i=j$ and $u\neq 0$) or ($i\neq j$ and $u\in\ZZ$) we have
            \begin{gather*}
                \tilde{c}_{ij}(u-r_i)+\tilde{c}_{ij}(u+r_i)-\sum_{\{k\in I\mid C_{ik=-1}\}}\tilde{c}_{kj}(u)-\sum_{\{k\in I\mid C_{ik=-2}\}}\Big(\tilde{c}_{kj}(u+1)+\tilde{c}_{kj}(u-1)\Big)-\\-\sum_{\{k\in I\mid C_{ik=-3}\}}\Big(\tilde{c}_{kj}(u-2)+\tilde{c}_{kj}(u)+\tilde{c}_{kj}(u+2)\Big)=0.
            \end{gather*}
\end{lemma}

\begin{proof}
  These identities follow from the fact that $C(z)$ and $\tilde{C}(z)$ are inverse of each other. Indeed, for every $i,j$ in $I$, they verify:
  \[\sum_{k=1}^nC_{ik}(z)\tilde{C}_{kj}(z)=\delta_{ij}.\]
  By Definition \ref{def quantum cartan}, we can rewrite this as
  \[(z^{r_i}+z^{-r_i})\tilde{C}_{ij}(z)+\sum_{\{k\in I\mid k\neq i\}}[C_{ik}]_z\tilde{C}_{kj}(z)=\delta_{ij}.\]
  We use the Laurent expansion of the coefficient of $\tilde{C}(z)$ to get:
  \begin{gather*}
      \sum_{u\in\ZZ}\Big(z^{r_i}\tilde{c}_{ij}(u)+z^{-r_i}\tilde{c}_{ij}(u)-\sum_{\{k\in I\mid C_{ik}=-1\}}\tilde{c}_{kj}(u)-\sum_{\{k\in I\mid C_{ik}=-2\}}(z+z^{-1})\tilde{c}_{kj}(u)-\\-\sum_{\{k\in I\mid C_{ik}=-3\}}(z^2+1+z^{-2})\tilde{c}_{kj}(u)\Big)z^u=\delta_{ij}.
  \end{gather*}
  This is equivalent to 
  \begin{gather*}
       \sum_{u\in\ZZ}\Big(\tilde{c}_{ij}(u-r_i)+\tilde{c}_{ij}(u+r_i)-\sum_{\{k\in I\mid C_{ik=-1}\}}\tilde{c}_{kj}(u)-\sum_{\{k\in I\mid C_{ik=-2}\}}\big(\tilde{c}_{kj}(u+1)+\tilde{c}_{kj}(u-1)\big)-\\\sum_{\{k\in I\mid C_{ik=-3}\}}\big(\tilde{c}_{kj}(u-2)+\tilde{c}_{kj}(u)+\tilde{c}_{kj}(u+2)\big)\Big)z^u=\delta_{ij}.
  \end{gather*}
  Comparing the two sides of the equation, we conclude the proof.
\end{proof}
Note that in the simply-laced types, relations in Lemma \ref{relation quantum cartan} appeared in \cite{b21}. The next Proposition gives some other useful properties about the inverse quantum Cartan matrix.
\begin{proposition}[{\cite[Lemma 3.4]{fo_qdata}}]\label{prop tecnica 2}
    For $i,j\in I$ and $u\in\ZZ$, we have
    \begin{enumerate}
         \item $\tilde{c}_{ij}(u+r_i)-\tilde{c}_{ij}(u-r_i)=\tilde{c}_{ji}(u+r_j)-\tilde{c}_{ji}(u-r_j),$\label{3}
        \item 
        \[\tilde{c}_{ij}(u)=\begin{cases}
            \tilde{c}_{ji}(u)\quad &\textit{if}\ r_i=r_j,\\
            \tilde{c}_{ji}(u+1)+\tilde{c}_{ji}(u-1)\quad &\textit{if}\ (r_i,r_j)=(1,2),\\
            \tilde{c}_{ji}(u-2)+\tilde{c}_{ji}(u)+\tilde{c}_{ji}(u+2)\quad &\textit{if}\ (r_i,r_j)=(1,3).\end{cases}\]\label{4}
    \end{enumerate}
\end{proposition}
Throughout the paper we fix a non-zero complex number $q$, which is not a root of unity, and we denote $q_i:=q^{r_i}$.

\subsection*{Bicharacters}
We will use the following notation:
\[V:=I\times \ZZ.\]
Let 
\[\mathcal{Y}:=\ZZ[Y_{i,q^r}^{\pm 1}\mid (i,r)\in V]\] be the ring of Laurent polynomials in the variables $Y_{i,q^r}$. Inside $\mathcal{Y}$ we can consider the multiplicative group of Laurent monomials $\mathcal{M}$. We introduce now another ring of Laurent polynomials
\[\mathcal{Y}':=\ZZ[\Psi_{i,q^r}^{\pm 1}\mid (i,q^r)\in V]\] and we denote by $\mathcal{M}'$ the multiplicative group of Laurent monomials inside $\mathcal{Y}'$. By the following relation, we can consider $\mathcal{Y}$ as a subring of $\mathcal{Y}'$ and $\mathcal{M}$ as a subgroup of $\mathcal{M}'$:
\begin{equation}\label{rel Y e psi}
    Y_{i,q^{r}}:=\Psi_{i,q_i^{-1}q^{r}}\Psi_{i,q_iq^{r}}^{-1}=\Psi_{i,q^{r-r_i}}\Psi_{i,q^{r+r_i}}^{-1}.
\end{equation}
Note that $\mathcal{Y}'$ and $\mathcal{Y}$ ( resp. $\mathcal{M}'$ and $\mathcal{M}$) are commutative rings (resp. groups). 

The reason to deal with these two rings and the relation between them comes from the representation theory of quantum affine algebras, as mentioned in the \hyperlink{target:intro}{Introduction}.

Now we give two crucial definitions for our work.
\begin{definition}\label{def N}
   Let $$\Omega:\mathcal{M}\times\mathcal{M}\to \ZZ$$ be the bicharacter defined on the group generators by
    \begin{equation}\label{eq def N}
        \Omega(Y_{i,q^r},Y_{j,q^s})=\tilde{c}_{ij}(m+r_i)-\tilde{c}_{ij}(m-r_i)+\tilde{c}_{ij}(-m-r_i)-\tilde{c}_{ij}(-m+r_i),
    \end{equation}
    where $m=s-r$.
\end{definition}
The definition of $\Omega$ was first given in \cite{halgebraicqt}, using the notation $\mathcal{N}$. 
\begin{remark}\label{N antisimmetrica}
    By Proposition \ref{prop tecnica} (\ref{3}), $\Omega$ verifies
    \[\Omega(Y_{i,q^r},Y_{j,q^s})=-\Omega(Y_{j,q^s},Y_{i,q^r}),\quad\forall(i,r),(j,s)\in V.\] Thus, $\Omega$ is skew-symmetric.
\end{remark}
\begin{remark}
    By Proposition \ref{prop tecnica}, $\tilde{c}_{ij}(k)=0$ if $k<r_i$, so if $m\geq 0$, the Equation \eqref{eq def N} reduces to
    \begin{equation}
       \Omega(Y_{i,q^r},Y_{j,q^s}) =\tilde{c}_{ij}(m+r_i)-\tilde{c}_{ij}(m-r_i).
    \end{equation}
\end{remark}
Now we extend the definition of $\Omega$ to the direct product of groups $\mathcal{M}'$. Note that this has be done in \cite{b21} for simply-laced finite Dynkin types. We propose a definition for all finite Dynkin types. 
\begin{definition}\label{def F}
    Let $\Omega':\mathcal{M}'\times\mathcal{M}'\to\ZZ$ be the bicharacter defined on the group generators by
    
    \begin{equation}\label{def F general}
       \Omega'(\Psi_{i,q^r},\Psi_{j,q^s}):=-\sum_{\ell \geq 1}\tilde{c}_{ij}(m+r_j(-2\ell+1))+\sum_{\ell\geq 1}\tilde{c}_{ij}(-m+r_j(-2\ell+1)),
    \end{equation}
    where $m=s-r$.
\end{definition}
\begin{remark}\label{F recovers bittmann}
    \begin{enumerate}[label=\roman*)]
        \item By Proposition \ref{prop tecnica} we have that the summation in \eqref{def F general} is finite and hence well defined. In particular, for $m = s - r\geq 0$, it reduces to
        \[\Omega'(\Psi_{i,q^r},\Psi_{j,q^s}):=-\sum_{\ell\geq 1}\tilde{c}_{ij}(m+r_j(-2\ell+1)).\]
        \item Note that in the simply-laced case we have $r_j=1$, so we recover the definition given in \cite{b21} (in that paper, the bicharacter is denoted by $\mathcal{F}$).\\
        \item Eventually, $\Omega'(\Psi_{i,q^r},\Psi_{j,q^s})$ depends only on the difference $m = s-r$ of the parameters, and on $i,j$.
    \end{enumerate}
\end{remark}
\begin{proposition}\label{F antisimmetrica}
    The bicharacter $\Omega'$ is skew-symmetric.
\end{proposition}
\begin{proof}
    The proof is based on Proposition \ref{prop tecnica}(\ref{4}). Indeed, when $r_i=r_j$, we can conclude immediately. When $r_i=1,r_j=2$, denoting as above $m=s-r$, we have 
    \begin{align*}
       \Omega'(\Psi_{i,q^r},\Psi_{j,q^s})&=-\sum_{\ell\geq 1}\tilde{c}_{ij}(m+2(-2\ell+1))+\sum_{\ell\geq 1}\tilde{c}_{ij}(-m+2(-2\ell+1))=\\
        &=-\sum_{\ell\geq 1}\tilde{c}_{ij}(m-4\ell+2)+\sum_{\ell\geq 1}\tilde{c}_{ij}(-m-4\ell+2).
    \end{align*} Using the relation from Proposition \ref{prop tecnica 2}(\ref{4}) with $u=\pm m-4\ell+2$, we get
    \begin{align*}
       \Omega'(\Psi_{i,q^r},\Psi_{j,q^s})&=-\sum_{\ell\geq 1}\tilde{c}_{ji}(m-4\ell+3)-\sum_{\ell\geq 1}\tilde{c}_{ji}(m-4\ell+1)+\sum_{\ell\geq 1}\tilde{c}_{ji}(-m-4\ell+3)+\sum_{\ell\geq 1}\tilde{c}_{ji}(-m-4\ell+1)=\\
        &=-\sum_{\ell\geq 1}\tilde{c}_{ji}(m-2\ell+1)+\sum_{\ell\geq 1}\tilde{c}_{ji}(-m-2\ell+1)=\\
        &=-\left(-\sum_{\ell\geq 1}\tilde{c}_{ji}(-m+(-2\ell+1))+\sum_{\ell\geq 1}\tilde{c}_{ji}(m+(-2\ell+1))\right)\\
        &=-\Omega'(\Psi_{j,q^s},\Psi_{i,q^r}).
    \end{align*}
    Finally, the case $r_i=1, r_j=3$ can be treated similarly.
\end{proof}

\begin{lemma}\label{restrizione Lambda}
    The restriction of $\Omega'$ to $\mathcal{M}\times\mathcal{M}$ coincides with $\Omega$. 
\end{lemma}
\begin{proof}
    We have to show that for every $(i,r),(j,s)$ in $V$, we have $\Omega'(Y_{i,q^r},Y_{j,q^s})=\Omega(Y_{i,q^r},Y_{j,q^s})$. Using the relation in Equation \eqref{rel Y e psi}, we have to prove 
    \[\Omega'(\Psi_{i,q^{r-r_i}},\Psi_{i,q^{r+r_i}}^{-1},\Psi_{j,q^{s-r_j}}\Psi_{j,q^{s+r_j}}^{-1})=\Omega(Y_{i,q^r},Y_{j,q^s}).\] Since $\Omega'$ is a bicharacter, this is equivalent to prove:
    \begin{equation}\label{aaa}
        \Omega'(\Psi_{i,q^{r-r_i}},\Psi_{j,q^{s-r_j}})-\Omega'\left(\Psi_{i, q^{r - r_i}}, \Psi_{j, q^{s + r_j}}\right) -\Omega'\left(\Psi_{i, q^{r + r_i}}, \Psi_{j, q^{s - r_j}}\right) +\Omega'\left(\Psi_{i, q^{r + r_i}}, \Psi_{j, q^{s + r_j}}\right)=\Omega(Y_{i,q^r},Y_{j,q^s})
    \end{equation}
    Developing the left hand side of \eqref{aaa}, setting $m=s-r$, we get:
    \begin{align*}
        \mathrm{LHS}=&-\sum_{\ell\geq 1}\tilde{c}_{ij}(m-r_j+r_i+r_j(-2\ell+1))+\sum_{\ell\geq1}\tilde{c}_{ij}(-m+r_j-r_i+r_j(-2\ell+1))+\\
        &+\sum_{\ell\geq 1}\tilde{c}_{ij}(m+r_j+r_i+r_j(-2\ell+1))-\sum_{\ell\geq 1}\tilde{c}_{ij}(-m-r_j-r_i+r_j(-2\ell+1))+\\
        &+\sum_{\ell\geq 1}\tilde{c}_{ij}(m-r_j-r_i+r_j(-2\ell+1))-\sum_{\ell\geq 1}\tilde{c}_{ij}(-m+r_j+r_i+r_j(-2\ell+1))-\\
        &-\sum_{\ell\geq 1}\tilde{c}_{ij}(m+r_j-r_i+r_j(-2\ell+1))+\sum_{\ell\geq 1}\tilde{c}_{ij}(-m-r_j+r_i+r_j(-2\ell+1)).
    \end{align*}
    This sum is telescopic and the remaining terms are
    \[\tilde{c}_{ij}(m+r_i)-\tilde{c}_{ij}(m-r_i)+\tilde{c}_{ij}(-m-r_i)-\tilde{c}_{ij}(-m+r_i),\]and this is indeed equal to $\Omega(Y_{i,q^r},Y_{j,q^s})$.
    
\end{proof}
\begin{remark}
    If we consider the groups $\mathcal{M}$ and $\mathcal{M}'$, ignoring the relation \eqref{rel Y e psi}, we have that via $Y_{i,q^r}\mapsto\Psi_{i,q^r}$ the two groups are isomorphic. However, the bicharacters $\Omega$ and $\Omega'$ are not identified through this isomorphism. For instance, in type $A_2$ we have $\Omega(Y_{1,q^0},Y_{2,q^5})=1$, while $\Omega'(\Psi_{1,q^0},\Psi_{2,q^5})=0$.
\end{remark}

\section{Braid group invariance of the bicharacter}\label{sec braid}

We recall Chari's braid group action on the group $\mathcal{M}$ and Frenkel--Hernandez's extension of this action to $\mathcal{M}'$. Then we prove our first main Theorem: the bicharacter $\Omega'$ is invariant under the braid group action (Theorem \ref{main thm}). Finally we give an example to illustrate the Theorem.

Let $W$ be the Weyl group of $\mathfrak{g}$ and $\mathcal{B}$ the corresponding braid group. $\mathcal{B}$ is generated by elements $(t_i)_{i\in I}$ and relations 
\begin{align*}
    t_it_j&=t_jt_i\quad \textit{if}\ C_{ij}=0,\\
    t_it_jt_i&=t_jt_it_j\quad \textit{if}\ C_{ij}C_{ji}=1,\\
    (t_it_j)^2&=(t_jt_i)^2\quad \textit{if}\ C_{ij}C_{ji}=2,\\
    (t_it_j)^3&=(t_jt_i)^3\quad \textit{if}\ C_{ij}C_{ji}=3.
\end{align*}
Thus, the braid group $\mathcal{B}$ has infinite order. 

$\mathcal{B}$ acts on $\mathcal{\mathcal{M}}$ and $\mathcal{M}'$. The action on $\mathcal{M}$ was defined in \cite{chari}\footnote{We use here a slight modification of Chari's action, where $q$ is replaced by its inverse $q^{-1}$.} and works as follows:
\begin{definition}[\cite{chari}]\label{def braid Y}
    For every $i\in I$, we define an automorphism $T_i$ of $\mathcal{M}$ defined by
    \begin{align*}\label{eq def braid on Y}
        T_i(Y_{j,q^r})&:=Y_{j,q^r},\quad \textit{if}\ i\neq j,\\
        T_i(Y_{i,q^r})&:=Y_{i,q^{r-2r_i}}^{-1}\prod_{\{k\in I\mid C_{ki}=-1\}}Y_{k,q^{r-r_i}}\prod_{\{k\in I\mid C_{ki}=-2\}}Y_{k,q^{r-r_i-1}}Y_{k,q^{r-r_i+1}}\prod_{\{k\in I\mid C_{ki}=-3\}}Y_{k,q^{r-r_i-2}}Y_{k,q^{r-r_i}}Y_{k,q^{r-r_i+2}}.
        \end{align*}   
     
\end{definition}
It was established in \cite{chari} that the automorphisms $(T_i)_{i\in I}$ generate a braid group action on $\mathcal{M}$. For any $w\in W$, by considering a reduced expression, one has a corresponding braid group operator $T_w$ on $\mathcal{M}$. 

We have seen above that $\mathcal{M}$ is a subgroup of $\mathcal{M}'$. In \cite{fh23} the braid group action is extended to $\mathcal{M}'$ in the following way.

\begin{definition}[\cite{fh23}]\label{braid action}
    For every $i\in I$, $T_i'$ defines\footnote{Our definition differs from \cite[Eq. (3.9)]{fh23} replacing $q$ by $q^{-1}$ in the action. This yields an isomorphic braid group action. In addition, our definition is the same as in \cite[Remark 5.12]{ghl}, allowing us to treat $QQ$-systems in Section \ref{section QQ}.} an automorphism of $\mathcal{M}'$ by:
   \begin{align*}
       T_i'(\Psi_{j,q^r})&:=\Psi_{j,q^r}\quad\textit{if}\ i\neq j,\\
       T_i'(\Psi_{i,q^r})&:=\Psi_{i,q^{r-2r_i}}^{-1}\prod_{\{k\in I\mid C_{ik}=-1\}}\Psi_{k,q^{r-r_i}}\prod_{\{k\in I\mid C_{ik}=-2 \}}\Psi_{k,q^r}\Psi_{k,q^{r-2}}\prod_{\{k\in I\mid C_{ik}=-3\}}\Psi_{k,q^{r-1}}\Psi_{k,q^{r+1}}\Psi_{k,q^{r-3}}.
    \end{align*}
\end{definition}

By linearity, for every $i\in I$ we can extend the automorphisms $T_i'$ (resp. $T_i$) to $\mathcal{Y}'$ (resp. $\mathcal{Y}$). We have that: 
\begin{theorem}\label{thm 3.4 fh}\cite[Theorem 3.4]{fh23}
    The automorphisms $(T_i')_{i\in I}$ satisfy the braid relations, hence they define a braid group action on $\mathcal{Y}'$. Moreover, the restriction of $T_i'$ to $\mathcal{Y}$ coincides with Chari operator $T_i$.
\end{theorem}
For any reduced expression $w\in W$, we will denote by $T_w'$ the corresponding braid group operator.
\begin{remark}
    Notice that the definitions of the actions of $\mathcal{B}$ on $\mathcal{M}$ and on $\mathcal{M}'$ are essentially the same, except that the indices under the product symbols are exchanged. This could be an instance on Langlands duality between the variables $Y$'s and $\Psi$'s as discussed in other contexts \cite[Remark 3.7]{fh23}.
\end{remark}
Now we can state our main result.
\begin{theorem}\label{main thm}
    The bicharacter $\Omega'$ is $\mathcal{B}$-invariant, in the sense that for every $w$ in $W$ and every $M,N$ monomials in $\mathcal{M}'$, we have
    \begin{equation*}
        \Omega'\left(T'_w(M),T'_w(N)\right)=\Omega'(M,N).
    \end{equation*}
\end{theorem}
\begin{proof} 
It is enough to prove the statement on the generators of $\mathcal{M}$' and $\mathcal{B}$, that is for every $i,j,k\in I$ and every $r,s\in\ZZ$, we have to prove:
\begin{equation}\label{eq thm}
    \Omega'\left(T'_k(\Psi_{i,q^r}),T'_k(\Psi_{j,q^s})\right)=\Omega'(\Psi_{i,q^r},\Psi_{j,q^s}).
\end{equation}
    If we fix $(i,r),(j,s)$ in $V$, the result is immediately true for every $T'_k$ with $i\neq k\neq j$, by Definition \ref{braid action}. Thus we are left to prove two cases: the one with $i=k\neq j$ and the one with  $i=j=k$. In the first case we have that proving Equation \eqref{eq thm} is equivalent to prove that
    \[\Omega'\left(T'_i(\Psi_{i,q^r}),T'_i(\Psi_{j,q^s})\right)=\Omega'(\Psi_{i,q^r},\Psi_{j,q^s}).\]
    Then, by definition of $T_i'$ and by the properties of $\Omega'$, this is equivalent to:
\begin{equation}
\begin{split}
&-\Omega'(\Psi_{i,q^{r-2r_i}},\Psi_{j,q^s})+\sum_{\{k\in I\mid C_{ik}=-1\}}\Omega'(\Psi_{k,q^{r-r_i}},\Psi_{j,q^s})+\sum_{\{k\in I\mid C_{ik}=-2\}}\big(\Omega'(\Psi_{k,q^r},\Psi_{j,q^s})+\Omega'(\Psi_{k,q^{r-2}},\Psi_{j,q^s})\big)\\
&+\sum_{\{k\in I\mid C_{ik}=-3\}}\big(\Omega'(\Psi_{k,q^{r-1}},\Psi_{j,q^s})+\Omega'(\Psi_{k,q^{r+1}},\Psi_{j,q^s})+\Omega'(\Psi_{k,q^{r-3}},\Psi_{j,q^s})\big)-\Omega'(\Psi_{i,q^r},\Psi_{j,q^s})=0.
\end{split}
\end{equation}
    
    We set $m:=s-r$. Developing the LHS of the previous equation, using the definition of $\Omega'$, we obtain:
    \begin{gather*}
        \mathrm{LHS}=\sum_{\ell\geq 1}\Big(\tilde{c}_{ij}(m+2r_i+r_j(-2\ell+1))-\tilde{c}_{ij}(-m-2r_i+r_j(-2\ell+1))-\\
        -\sum_{\{k\in I\mid C_{ik}=-1\}}\tilde{c}_{kj}(m+r_i+r_j(-2\ell+1))+\sum_{\{k\in I\mid C_{ik}=-1\}}\tilde{c}_{kj}(-m-r_i+r_j(-2\ell+1))-\\
        -\sum_{\{k\in I\mid C_{ik}=-2\}}\tilde{c}_{kj}(m+r_j(-2\ell+1))+\sum_{\{k\in I\mid C_{ik}=-2\}}\tilde{c}_{kj}(-m+r_j(-2\ell+1))-\\
        -\sum_{\{k\in I\mid C_{ik}=-2\}}\tilde{c}_{kj}(m+2+r_j(-2\ell+1))+\sum_{\{k\in I\mid C_{ik}=-2\}}\tilde{c}_{kj}(-m-2+r_j(-2\ell+1))-\\
        -\sum_{\ell\geq 1}\sum_{\{k\in I\mid C_{ik}=-3\}}\tilde{c}_{kj}(m+1+r_j(-2\ell+1))+\sum_{\{k\in I\mid C_{ik}=-3\}}\tilde{c}_{kj}(-m-1+r_j(-2\ell+1))-\\
        -\sum_{\{k\in I\mid C_{ik}=-3\}}\tilde{c}_{kj}(m-1+r_j(-2\ell+1))+\sum_{\{k\in I\mid C_{ik}=-3\}}\tilde{c}_{kj}(-m+1+r_j(-2\ell+1))\\
        -\sum_{\{k\in I\mid C_{ik}=-3\}}\tilde{c}_{kj}(m+3+r_j(-2\ell+1))+\sum_{\{k\in I\mid C_{ik}=-3\}}\tilde{c}_{kj}(-m-3+r_j(-2\ell+1))-\\
        +\tilde{c}_{ij}(m+r_j(-2\ell+1))-\tilde{c}_{ij}(-m+r_j(-2\ell+1))\Big).
    \end{gather*}
    To conclude that the sum above is equal to $0$, we will use Lemma \ref{relation quantum cartan}. More precisely, if the sets $\{k\in I\mid C_{ik}=-2\}$ and $\{k\in I\mid C_{ik}=-3\}$ are both empty, then for each $\ell\geq 1$ we apply Lemma \ref{relation quantum cartan} with $u=m+r_i+r_j(-2\ell+1)$ (for the left column of sums) and with $u=-m-r_i+r_j(-2\ell+1)$ (for the right column of sums). Otherwise, if $i$ is such that for some $k$ in $I$, $C_{ik}<-1$, then by Remark \ref{remark sugli r_i} $r_i=1$, and for each $\ell\geq1$ we apply Lemma \ref{relation quantum cartan} with $u=m+1+r_j(-2\ell+1)$ and $u=-m-1+r_j(-2\ell+1)$. We now examine the case $i=j=k$. In this case we have to prove that for every $r,s\in \ZZ$:
        \begin{equation}\label{main claim i=j=k}
        \Omega'(T_i'(\Psi_{i,q^r}),T_i'(\Psi_{i,q^s}))=\Omega'(\Psi_{i,q^r},\Psi_{i,q^s}).
        \end{equation}
    For convenience we will use the following notation, for every $p\in \ZZ$: 
    \[\Delta_i(p):=\prod_{\{k\in I\mid C_{ik}=-1\}}\Psi_{k,q^{p-r_i}}\prod_{\{k\in I\mid C_{ik}=-2 \}}\Psi_{k,q^p}\Psi_{k,q^{p-2}}\prod_{\{k\in I\mid C_{ik}=-3\}}\Psi_{k,q^{p-1}}\Psi_{k,q^{p+1}}\Psi_{k,q^{p-3}},\] so that we can write\[T_i'(\Psi_{i,q^p})=\Psi_{i,q^{p-2r_i}}^{-1}\Delta_i(p).\]
    Observe that 
    \[\Omega'(\Psi_{i,q^{r-2r_i}},\Psi_{i,q^{s-2r_i}})=\Omega'(\Psi_{i,q^r},\Psi_{i,q^s}),\]
    then, developing the LHS of \eqref{main claim i=j=k}, we have that \eqref{main claim i=j=k} is equivalent to
    \begin{equation}\label{second i=j=k}
        \Omega'(\Delta_i(r),\Delta_i(s))=\Omega'(\Psi_{i,q^{r-2r_i}},\Delta_i(s))+\Omega'(\Delta_i(r),\Psi_{i,q^{s-2r_i}}).
    \end{equation}
    From the first half of the proof, we can deduce that for every $k\neq i$, we have
    \begin{equation}\label{delta}
        \Omega'(\Delta_i(r),\Psi_{k,q^s})=\Omega'(\Psi_{i,q^r},\Psi_{k,q^s})+\Omega'(\Psi_{i,q^{r-2r_i}},\Psi_{k,q^s}).
    \end{equation}
    Note that all the elements appearing in $\Delta_i(s)$ are variables of the form $\Psi_{k,q^p}$, with $k\neq i$. Thus, applying the properties of the bicharacter and \eqref{delta} (with $\Delta_i(s)$ replacing $\Psi_{k,q^s}$), we get
    \[\Omega'(\Delta_i(r),\Delta_{i}(s))=\Omega'(\Psi_{i,q^r},\Delta_i(s))+\Omega'(\Psi_{i,q^{r-2r_i}},\Delta_i(s)).\]
    Substituting this in the LHS of \eqref{second i=j=k}, we see that \eqref{second i=j=k} is equivalent to  
    \begin{equation}\label{final i=j=k}
        \Omega'(\Psi_{i,q^r},\Delta_i(s))+\Omega'(\Psi_{i,q^{s-2r_i}},\Delta_i(r))=0.
    \end{equation}
    If the fixed index $i$ is such that for every $k\neq i$ the only possible value for $C_{ik}$ is $-1$, then we have
    \begin{equation}\label{uno}
        \Omega'(\Psi_{i,q^r},\Delta_i(s))=\sum_{\{k\in I\mid C_{ik}=-1\}}\Omega'(\Psi_{i,q^r},\Psi_{k,q^{s-r_i}}),
    \end{equation}
    and
    \begin{equation}\label{due}
        \Omega'(\Psi_{i,q^{s-2r_i}},\Delta_i(r))=\sum_{\{k\in I\mid C_{ik}=-1\}}\Omega'(\Psi_{i,q^{s-2r_i}},\Psi_{k,q^{r-r_i}}).
    \end{equation} 
   By Definition \ref{def F} of $\Omega'$ we have that 
    \[\sum_{\{k\in I\mid C_{ik}=-1\}}\Omega'(\Psi_{i,q^r},\Psi_{k,q^{s-r_i}})=\sum_{\{k\in I\mid C_{ik}=-1\}}\Omega'(\Psi_{i,q^{s-2r_i}},\Psi_{k,q^{r-r_i}})\]
    so we conclude that \eqref{final i=j=k} holds.
    If, on the other hand, there exists an index $k\neq i$ such that $C_{ik}<-1$, then $r_i=1$ and we can argue similarly to conclude that \eqref{final i=j=k} holds also in this case. This concludes the proof.
\end{proof}

We get an analog and new result also for $\mathcal{M}$.

\begin{corollary}
    The bilinear form $\Omega$ is $\mathcal{B}$-invariant.
\end{corollary}
\begin{proof}
    This follows from Lemma \ref{restrizione Lambda} and Theorem \ref{thm 3.4 fh}.
\end{proof}
\begin{example}
    We compute an example in type $B_2$. The Cartan matrix is
    \[C(B_2)=\begin{pmatrix}
        2 & -1\\
        -2 & 2
    \end{pmatrix}\]
    and $r_1=2$, $r_2=1$. The quantum Cartan matrix and its inverse are
    \[C(z)=\begin{pmatrix}
        z+z^{-2}&  -1\\
        -(z+z^{-1})& z+z^{-1}
    \end{pmatrix},\quad \tilde{C}(z)=\frac{1}{z^3+z^{-3}}\begin{pmatrix}
        z+z^{-1} & 1\\
        z+z^{-1} & z^2+z^{-2}
    \end{pmatrix}.\]
    Note that, applying Lemma \ref{relation quantum cartan} with $i=2,j=1$, we get for every $u\in \ZZ$:
    \begin{equation}
        \tilde{c}_{21}(u-1)+\tilde{c}_{21}(u+1)=\tilde{c}_{11}(u-1)+\tilde{c}_{11}(u+1).
    \end{equation}
    This is equivalent to 
    \[(\tilde{c}_{21}-\tilde{c}_{11})(u+1)=(\tilde{c}_{21}-\tilde{c}_{11})(u-1).\] Since for $u\leq0$ the RHS is zero, the recursive formula let us conclude that
    \begin{equation}\label{uguaglianza c in B2}
        \tilde{c}_{21}(u)=\tilde{c}_{11}(u), \forall u\in\ZZ.
    \end{equation}
    Many other details and properties about the quantum Cartan datum in type $B_2$ can be found in \cite[Example 4.2, Appendix A]{ho_quantumclusterKL}.
    Using this, we have the following values of the coefficient $\tilde{c}_{ij}(u)$, for every $u\in \ZZ$:
    \begin{equation*}
        \begin{array}{llll}

        \tilde{c}_{11}(2u+1)=\tilde{c}_{21}(2u+1)=0 &\tilde{c}_{22}(2u)=0 & \tilde{c}_{12}(2u)=0 \\
        \tilde{c}_{11}(2)=\tilde{c}_{21}(2)=1 &\tilde{c}_{22}(1)=1 &\tilde{c}_{12}(1)=0 &\\
        \tilde{c}_{11}(4)=\tilde{c}_{21}(4)=1  &\tilde{c}_{22}(3)=0 &\tilde{c}_{12}(3)=1 &\\
        \tilde{c}_{11}(6)=\tilde{c}_{21}(6)=0 &\tilde{c}_{22}(5)=1 &\tilde{c}_{12}(5)=0 &\\
        \tilde{c}_{11}(8)=\tilde{c}_{21}(8)=-1 &\tilde{c}_{22}(7)=-1 &\tilde{c}_{12}(7)=0 &\\
        \tilde{c}_{11}(10)=\tilde{c}_{21}(10)=-1 &\tilde{c}_{22}(9)=0 &\tilde{c}_{12}(9)=-1 &\\
        \vdots & \vdots & \vdots &
    \end{array}
    \end{equation*}
    We have that 
    \[\Omega(Y_{1,q^0}, Y_{1,q^2})=\tilde{c}_{11}(2+r_1)=1.\]
    On the other hand, if we apply the braid operator $T_1$ we get
    \[T_1(Y_{1,q^0})=Y_{1,q^{-2}}^{-1}Y_{2,q^{-1}}Y_{2,q^{-3}},\quad T_1(Y_{1,q^2})=Y_{1,q^{0}}^{-1}Y_{2,q}Y_{2,q^{-1}}.\] Thus we have
    \begin{gather*}
        \Omega(T_1(Y_{1,q^0}),T_1(Y_{1,q^2}))=1,
    \end{gather*}
    which is equal to $\Omega(Y_{1,q^0}, Y_{1,q^2})$.
\end{example}

\section{Automorphisms of quantum tori}\label{sec q tori}
As a first application of our main result, we define a quantization $\mathcal{Y}'_t$ of $\mathcal{Y}'$ (Definition \ref{definew}), which generalizes to non simply-laced types the construction of \cite{b21}. Combining with Theorem \ref{main thm} we also prove that there exists a braid group action by automorphisms on this quantum torus (Theorem \ref{teo q tori}).

Let $t$ be a fixed formal variable.
\begin{definition}[\cite{halgebraicqt}]
   The quantum torus $(\mathcal{Y}_t,\ast)$ is the $\ZZ[t^{\pm \frac{1}{2}}]$-algebra generated by $\{Y_{i,q^r}^{\pm 1}\}_{(i,r)\in V}$, subject to the $t$-commutation relations
   \[Y_{i,q^r}\ast Y_{j,q^s}=t^{\Omega(Y_{i,q^r},Y_{j,q^s})}Y_{j,q^s}\ast Y_{i,q^r}.\]
  
\end{definition}

\begin{definition}\label{definew}
    The quantum torus $(\mathcal{Y}'_t,\star)$ is the $\ZZ[t^{\pm \frac{1}{2}}]$-algebra generated by $\{\Psi_{i,q^r}^{\pm 1}\}_{(i,r)\in V}$, subject to the $t$-commutation relations
   \[\Psi_{i,q^r}\star \Psi_{j,q^s}=t^{\Omega'(\Psi_{i,q^r},\Psi_{j,q^s})}\Psi_{j,q^s}\star \Psi_{i,q^r}.\]
   
\end{definition}

We can say that $\Omega(Y_{i,q^r},Y_{j,q^s})$ (resp. $\Omega'(\Psi_{i,q^r},\Psi_{j,q^s})$) "measures" how $Y_{i,q^r},Y_{j,q^s}$ (resp. $\Psi_{i,q^r},\Psi_{j,q^s}$) $t$-commute.

In simply-laced type, $\mathcal{Y}'_t$ was defined in \cite{b21} and contains a quantum torus defined by Nakajima \cite{nak04} and Varagnolo--Vasserot \cite{vv}.
Note that thank to Remark \ref{N antisimmetrica} and Proposition \ref{F antisimmetrica}, the relations in the definitions above are well defined.

When dealing with quantum tori, it is useful to consider the so-called \textit{commutative} monomials. Consider $(\mathcal{Y}'_t,\star)$ (the same works for $(\mathcal{Y}_t,\ast)$). Note that for every $(i,r), (j,s)$ in $V$, the expression
\[t^{\frac{1}{2}\Omega'(\Psi_{j,q^s},\Psi_{i,q^r})}\Psi_{i,q^r}\star \Psi_{j,q^s}=t^{\frac{1}{2}\Omega'(\Psi_{i,q^r},\Psi_{j,q^s})}\Psi_{j,q^s}\star\Psi_{i,q^r}\] is invariant under permutation of $(i,r)$ and $(j,s)$. Thus, we denote 
\[\Psi_{i,q^r}\Psi_{j,q^s}:=t^{\frac{1}{2}\Omega'(\Psi_{j,q^s},\Psi_{i,q^r})}\Psi_{i,q^r}\star \Psi_{j,q^s}
\in \mathcal{Y}_{t}'\] and we call it a \textit{commutative monomial}. We can generalize this concept to the product of several variables. A point to keep in mind is that the product $\star$ in $\mathcal{Y}'_t$ depends on the order. To solve this technicality, let us denote by $\overset{\rightarrow}{\star}$ the ordered product of elements in $\mathcal{Y}'_{t}$. Then, for a family of integers $(n_{(i,r)})_{(i,r)\in V}$ with only finitely many non-zero components, we define
\[\prod_{(i,r)\in V}\Psi_{i,q^r}^{n_{(i,r)}}:=t^{\frac{1}{2}\sum_{(i,r)<(j,s)}n_{(i,r)}n_{(j,s)}\Omega'(\Psi_{j,q^s},\Psi_{i,q^r})}\ \overset{\rightarrow}{\star}\Psi_{i,q^r}^{n_{(i,r)}}.\] Then, we can consider either the commutative or non-commutative product of commutative monomials. 

\begin{proposition}\label{inclusion quantum tori}
    There is an injective $\ZZ[t^{\pm 1/2}]$-algebras morphism $\mathcal{Y}_t\hookrightarrow \mathcal{Y}'_t$ defined by
    \[Y_{i,q^{r}}\mapsto\Psi_{i,q^{r-r_i}}\Psi_{i,q^{r+r_i}}^{-1},\] where the product on the right is commutative.
\end{proposition}
\begin{proof}
    We have to check that the $t$-commutation relations between two elements $Y_{i,q^r}$ and $Y_{j,q^s}$ in $\mathcal{Y}_t$ is the same as the $t$-commutation relation between their images in $\mathcal{Y}'_t$. In other words, recalling $\Omega$ and $\Omega'$ from Definition \ref{def N} and Definition \ref{def F}, we have to prove that
    \[\Omega(Y_{i,q^r},Y_{j,q^s})=\Omega'(\Psi_{i,q^{r-r_i}}\Psi_{i,q^{r+r_i}}^{-1},\Psi_{j,q^{s-r_j}}\Psi_{j,q^{s+r_j}}^{-1}),\] but this is true by Lemma \ref{restrizione Lambda}. The injectivity is clear as the commutative monomials form a $\mathbb{Z}[t^{\pm 1/2}]$-basis of $\mathcal{Y}_t$ and their images are linearly independent.
\end{proof}
The following Definition is natural:
\begin{definition}
    For every $i\in I$, let $T'_{t,i}:\mathcal{Y}'_t\to\mathcal{Y}'_t$ be the map defined on the generators as in Definition \ref{braid action}, where now the monomials on the right are meant to be commutative. Similarly, for every $i\in I$, let $T_{t,i}:\mathcal{Y}_t\to\mathcal{Y}_t$ be the map defined on the generators as in Definition \ref{def braid Y}, with the monomials now meant to be commutative.
\end{definition}

The following establishes that these maps are well-defined.

\begin{theorem}\label{teo q tori}
    For every $i\in I$, the maps $T'_{t,i}$ is an automorphism of the quantum torus $\mathcal{Y}'_{t}$. In particular, there is a braid group action on $\mathcal{Y}'_t$ generated by the automorphisms $\{T'_{t,i}\}_{i\in I}$. Moreover, the braid group action can be restricted to $\mathcal{Y}_t$, via the automorphisms $(T_{t,i})_{i\in I}$.
\end{theorem}
\begin{proof}
    The result follows from Theorem \ref{thm 3.4 fh} and Theorem \ref{main thm}, which guarantees that the $t$-relations in $\mathcal{Y}'_t$ are preserved by $T'_{t,i}$, for every $i\in I$. 
\end{proof}
\section{Quantum \texorpdfstring{$QQ$}{QQ}-systems}\label{section QQ}
\label{section QQ}
We give another application of our main result to quantum $QQ$-systems (Theorem \ref{preciseqQQ}). First, we give some reminders on $QQ$-systems and their cluster-theoretic role in \cite{ghl}. We recall the quantum cluster algebra $\mathcal{A}_t$ constructed in \cite{pag25} and its connection with our quantum torus $\mathcal{Y}'_t$. Then, we use Theorem \ref{main thm} to state explicitly the quantum $QQ$-systems in all simply-laced types. In particular, we establish that the power of $t$ that appear in the system is the same for all simply-laced types.

We need some more standard notation. For every $i\in I$, we denote by $s_i$ the simple reflection in the Weyl group $W$ and by $\omega_i$ the set of fundamental weights of $\mathfrak{g}$. Moreover, let $P$ be the integral weight lattice of $\mathfrak{g}$. In the following we will deal with quantum cluster algebras, non commutative versions of Fomin--Zelevinsky's cluster algebras. Since here we want to present an application of our main result, we skip the reminders on (quantum) cluster algebras and we refer the reader to the foundational papers \cite{fz1}, \cite{bz05}.

In \cite[Section 5]{fh23}, the authors established an important system of algebraic identities, called \textit{extended} $QQ$-\textit{systems}. These identities find their representation-theoretic meaning in the Grothendieck ring $K(\mathcal{O}^{\sh})$ of the category $\mathcal{O}^{\sh}$ of representations of shifted quantum affine algebras and hold for any underlying $\mathfrak{g}$. A renormalized version of this system is presented in \cite[Section 7.2]{ghl} for simply-laced types. Here the variables that satisfy the system are elements in a completion $\widetilde{\mathcal{Y}'_e}$ of $\mathcal{Y}'$ (this\footnote{More precisely, in \cite{ghl} $\widetilde{\mathcal{Y}'_e}$ is a completion of a scalar extension of $\mathcal{Y}'$ with the group ring of $P\otimes\QQ$.} and other completions, indexed by Weyl group elements, were defined first in \cite{fh22}) and are denoted $\underline{Q}_{w(\omega_i),q^r}$, for every $w\in W$ and $(i,r)\in V= I \times \mathbb{Z}$. Let us state the system and then comment more about it. 
\begin{proposition}[{\cite[Proposition 7.2]{ghl}}]
    For a fixed $\mathfrak{g}$ of simply-laced type, for $(i,r)\in V$ and $w\in W$ such that $ws_i>w$, we have
    \begin{equation}\label{classic QQ}
        \underline{Q}_{ws_i(\omega_i),q^r}\underline{Q}_{w(\omega_i),q^{r-2}}-\underline{Q}_{ws_i(\omega_i),q^{r-2}}\underline{Q}_{w(\omega_i),q^r}=\prod_{\{j:C_{ij}=-1\}}\underline{Q}_{w(\omega_j),q^{r-1}}.
    \end{equation}
\end{proposition}
A significant property of these identities is that in \cite{ghl} they serve as a subset of exchange relations in the cluster algebra structure on $K(\mathcal{O}^{\sh})$, for simply-laced $\mathfrak{g}$. Let us call $\mathcal{A}$ this cluster algebra. In particular, the variables $\underline{Q}_{w(\omega_i),q^r}$ can be interpreted as cluster variables in $\mathcal{A}$.
\begin{remark}
    $QQ$-systems hold also for non simply-laced types, however in this Section we are interested in their cluster-theoretic meaning. Thus, as for the moment the results in \cite{ghl} are available only for simply-laced types, we also restrict our attention to this setting.
\end{remark}
Although the $\underline{Q}$-variables are formal power series, their leading term with respect to a fixed partial order is a monomial in $\mathcal{M}'$ and can be computed. Let us denote by $\Psi_{w(\omega_i),q^r}$ the leading term of $\underline{Q}_{w(\omega_i),q^r}$. Then, from \cite[Proposition 3.6]{fh23}, we have:
\begin{proposition}\label{hw QQ}
    For every $w\in W$, $(i,r)\in V$, the leading term $\Psi_{w(\omega_i),q^r}$ of $\underline{Q}_{w(\omega_i),q^r}$ can be computed \footnote{In fact, the leading term of any $\underline{Q}_{w(\omega_i),q^r}$ has also a factor in $P\otimes\QQ$. However, it is not relevant for our computations and to make the text lighter we will omit it. See \cite[Section 7.2]{ghl} for details.} as
    \begin{equation}
    \Psi_{w(\omega_i),q^r}=T_w'(\Psi_{i,q^r}).
\end{equation}
\end{proposition}

In \cite{pag25} we constructed a quantization of the cluster algebra $\mathcal{A}$ of \cite{ghl}. Let $t$ be the deformation parameter as in the previous section and let $\mathcal{A}_t$ be such quantum cluster algebras. Since to every cluster variable of $\mathcal{A}$ corresponds a quantum cluster variable in $\mathcal{A}_t$, we have well-defined quantum $\underline{Q}$-variables in $\mathcal{A}_t$. By abuse of notation, we denote these quantum $\underline{Q}$-variables by $\underline{Q}_{w(\omega_i),q^r}$, as the classical ones. Let $\mathcal{T}$ be the quantum torus in which $\mathcal{A}_t$ lives. In \cite[Theorem 4.15]{pag25} we proved that the quantum tori $\mathcal{T}$ is a subring of $\mathcal{Y}'_t$. The same Theorem implies that in order to compute the $t$-relations between quantum $\underline{Q}$-variables in the same seed, we can consider the corresponding non-quantum $\underline{Q}$-variables and compute the $t$-relations between their leading terms, considered as commutative monomials in $\mathcal{Y}'_t$. Concretely, thanks to Proposition \ref{hw QQ}, for every $w,v\in W$, $ (i,r),(j,s)\in V$ such that $\underline{Q}_{w(\omega_i),q^r},\ \underline{Q}_{v(\omega_j),q^s}$ are in the same seed, if we want to compute $\alpha$ such that 
\[\underline{Q}_{w(\omega_i),q^r}\star\underline{Q}_{v(\omega_j),q^s}=t^{\alpha}\underline{Q}_{v(\omega_j),q^s}\star\underline{Q}_{w(\omega_i),q^r}\]
we can compute equivalently:
\[\alpha=\Omega'(\Psi_{w(\omega_i),q^r},\Psi_{v(\omega_j),q^s})=\Omega'(T_{t,w}'(\Psi_{i,q^r}),T'_{t,v}(\Psi_{j,q^s})).
\]
Thus, by abuse of notation, let us denote by $\Omega'(\underline{Q}_{w(\omega_i),q^r}\underline{Q}_{v(\omega_j),q^s})$ the power of $t$ that gives the $t$-commutation of two $\underline{Q}$-variables in the same seed.

In \cite[Section 5]{pag25} we explained that we could obtain a quantized version of the $QQ$-systems \eqref{classic QQ} as quantum exchange relations in $\mathcal{A}_t$. However, in our previous work we computed explicitly the quantum $QQ$-systems only for type $A_1$. In general case, we proved the existence, without computing them, of integers $A,B$ such that
\begin{equation}\label{AB qQQ}
  \underline{Q}_{w(\omega_i),q^{r-2}}\star \underline{Q}_{ws_i(\omega_i),q^r}= t^A\underline{Q}_{ws_i(\omega_i),q^{r-2}}\star\underline{Q}_{w(\omega_i),q^r} +t^B\underset{\{j\in I\mid C_{ij}=-1\}}{\overset{\rightarrow}{\star}} \underline{Q}_{w(\omega_j),q^{r-1}}
\end{equation} 
which is an equation in $\mathcal{A}_t$. The meaning of this relation is that $\underline{Q}_{w(\omega_i),q^{r-2}}$ is obtained from $\underline{Q}_{ws_i(\omega_i),q^r}$ by mutation. 
\begin{proposition}
    Let us fix $w\in W$ and $r\in \ZZ$. Then, every $\underline{Q}$-variable in the set $\Bigl\{\underline{Q}_{w(\omega_j),q^{r-1}}\Bigr\}_{\{j\in I,C_{ij}=-1\}}$ commute with each other. In particular, we can denote
    \[\underset{\{j\in I\mid C_{ij}=-1\}}{\overset{\rightarrow}{\star}} \underline{Q}_{w(\omega_j),q^{r-1}}=\prod_{\{j\in I\mid C_{ij}=-1\}} \underline{Q}_{w(\omega_j),q^{r-1}},\] for any choice of the ordered product.
\end{proposition}
\begin{proof}
    It suffices to prove that for any $h,k\in I$ such that $C_{ih}=C_{ik}=-1$ the elements $\underline{Q}_{w(\omega_h),q^{r-1}},\underline{Q}_{w(\omega_k),q^{r-1}}$ commute. By the discussion above, we have
    \[\Omega'\left(\underline{Q}_{w(\omega_h),q^{r-1}},\underline{Q}_{w(\omega_k),q^{r-1}}\right)=\Omega'\left(T_{t,w}(\Psi_{h,q^{r-1}}),T_{t,w}(\Psi_{k,q^{r-1}})\right).\] Then, by Theorem \ref{teo q tori} we can further simplify the expression and get
    \[\Omega'\left(\underline{Q}_{w(\omega_h),q^{r-1}},\underline{Q}_{w(\omega_k),q^{r-1}}\right)=\Omega'(\Psi_{h,q^{r-1}},\Psi_{k,q^{r-1}}),\] but this is zero by Definition \ref{def F} and Proposition \ref{prop tecnica}, since the spectral parameters that appear are the same.
\end{proof}
As a consequence, we can rewrite \eqref{AB qQQ} as:
\begin{equation}\label{ABqq2}
    \underline{Q}_{w(\omega_i),q^{r-2}}\star \underline{Q}_{ws_i(\omega_i),q^r}= t^A\underline{Q}_{ws_i(\omega_i),q^{r-2}}\star\underline{Q}_{w(\omega_i),q^r} +t^B\prod_{\{j\in I\mid C_{ij}=-1\}} \underline{Q}_{w(\omega_j),q^{r-1}} .
\end{equation}
In what follows we apply the same argument to compute $A,B$.
Since \eqref{ABqq2} is a quantum exchange relation, from the theory of quantum cluster algebras (and toric frames), $A,B$ are computed as follows:
\begin{align*}
    A&=-\frac{1}{2}\Omega'\Big(\underline{Q}_{w(\omega_i),q^{r-2}}^{-1},\underline{Q}_{ws_i(\omega_i),q^{r-2}}\Big)-\frac{1}{2}\Omega'\Big(\underline{Q}_{w(\omega_i),q^{r-2}}^{-1},\underline{Q}_{w(\omega_i),q^r}\Big)-\frac{1}{2}\Omega'\Big(\underline{Q}_{ws_i(\omega_i),q^{r-2}},\underline{Q}_{w(\omega_i),q^r}\Big),\\
    B&=-\frac{1}{2}\Omega'\Big(\underline{Q}_{w(\omega_i),q^{r-2}}^{-1},\prod_{\{j\in I\mid C_{ij}=-1\}}\underline{Q}_{w(\omega_j),q^{r-1}}\Big).
\end{align*}
Recalling Proposition \ref{hw QQ} and that we are regarding now the leading terms as commutative monomials in $\mathcal{Y}'_t$, we can rewrite
\begin{align*}
    A&=\frac{1}{2}\Omega'(T_{t,w}'(\Psi_{i,q^{r-2}}),T_{t,w}'T_{t,i}'(\Psi_{i,q^{r-2}}))+\frac{1}{2}\Omega'(T_{t,w}'(\Psi_{i,q^{r-2}}),T_{t,w}'(\Psi_{i,q^r}))-\frac{1}{2}\Omega'(T_{t,w}'T_{t,i}'(\Psi_{i,q^{r-2}}),T_{t,w}'(\Psi_{i,q^r})),\\
    B&=\frac{1}{2}\Omega'\Big(T_{t,w}'(\Psi_{i,q^{r-2}}),\prod_{\{j\in I\mid C_{ij}=-1\}}T_{t,w}'(\Psi_{j,q^{r-1}})\Big).
\end{align*}

By Theorem \ref{teo q tori} we can cancel out the operators $T_{t,w}'$, and we get
\begin{align*}
    A&=\frac{1}{2}\Omega'(\Psi_{i,q^{r-2}},T_i'(\Psi_{i,q^{r-2}}))+\frac{1}{2}\Omega'(\Psi_{i,q^{r-2}},\Psi_{i,q^r})-\frac{1}{2}\Omega'(T_i'(\Psi_{i,q^{r-2}}),\Psi_{i,q^r}),\\
    B&=\frac{1}{2}\Omega'\Big(\Psi_{i,q^{r-2}}
    ,\prod_{\{j\in I\mid C_{ij}=-1\}}\Psi_{j,q^{r-1}}\Big).
\end{align*}
Since now we are working in simply-laced setting, the braid operator acts in a simpler way than in general:

\begin{equation*}
    A=\frac{1}{2}\Omega'\left(\Psi_{i,q^{r-2}},\Psi_{i,q^{r-4}}^{-1}\prod_{\{j\in I\mid C_{ij}=-1\}}\Psi_{k,q^{r-3}}\right)+\frac{1}{2}\Omega'\Big(\Psi_{1,q^{r-2}},\Psi_{i,q^r}\Big)-\frac{1}{2}\Omega'\left(\Psi_{i,q^{r-4}}^{-1}\prod_{\{j\in I\mid C_{ij}=-1\}}\Psi_{k,q^{r-3}},\Psi_{i,q^r}\right).
\end{equation*}

Recalling Definition \ref{def F}, we have
\begin{equation*}
    \begin{split}
        A = \frac{1}{2} \Big[ & -\Omega'(\Psi_{i,q^{r-2}}, \Psi_{i,q^{r-4}}) + \Omega'(\Psi_{i,q^{r-2}}, \Psi_{i,q^r}) + \Omega'(\Psi_{i,q^{r-4}}, \Psi_{i,q^r}) \\
& + \sum_{\{k\in I\mid C_{ik}=-1\}} \left( \Omega'(\Psi_{i,q^{r-2}}, \Psi_{k,q^{r-3}}) - \Omega'(\Psi_{k,q^{r-3}}, \Psi_{i,q^r}) \right) \Big]
    \end{split}
\end{equation*}

Using the properties of $\Omega'$ and those of the quantum Cartan matrix from Section \ref{sec quantum cartan}, we finally get
\begin{equation}\label{final A}
    A=-\frac{3}{2}-\frac{\tilde{c}_{ii}(3)}{2}+\sum_{\{k\in I\mid C_{ik}=-1\}}\frac{\tilde{c}_{ki}(2)}{2}.
\end{equation}
Similarly for $B$:
\begin{equation}
    B=\frac{1}{2}\sum_{\{j\in I\mid C_{ij}=-1\}}\Omega'(\Psi_{i,q^{r-2}},\Psi_{j,q^{r-1}})=0.
\end{equation}
These considerations yeld the following theorem.
\begin{theorem}\label{preciseqQQ}
    For $(i,r)\in V$ and $w\in W$ such that $ws_i>w$, the following identity holds in $\mathcal{A}_t$:
    \begin{equation}
        \underline{Q}_{w(\omega_i),q^{r-2}}\star\underline{Q}_{ws_i(\omega_i),q^r}= t^A\underline{Q}_{ws_i(\omega_i),q^{r-2}}\star\underline{Q}_{w(\omega_i),q^r} +\prod_{\{j\in\mid C_{ij}=-1\}}\underline{Q}_{w(\omega_j),q^{r-1}},
    \end{equation}
    where $A$ is determined by \eqref{final A}.
\end{theorem}

\begin{example}
    We can write all the quantum $QQ$-systems for type $A_2$. Then, for every $w$ in $W$, $(i,r)$ in $V$, $I=\{i,j\}$, we have
    \begin{equation}
        \underline{Q}_{w(\omega_i),q^{r-2}}\star\underline{Q}_{ws_i(\omega_i),q^r}=t^{-2}\underline{Q}_{ws_i(\omega_i),q^{r-2}}\star\underline{Q}_{w(\omega_i),q^r}+\underline{Q}_{w(\omega_j),q^{r-1}}.
    \end{equation}
    Where the values of $\tilde{c}_{ij}(u)$ for every $u$ in $\ZZ$ can be found in \cite[Example 3.2.3]{b21}.
\end{example}
\section{Compatibility with cluster algebras}\label{sec q cluster}
In this Section we present an implication of our definition of the bicharacter $\Omega'$ in the context of monoidal categorification  of cluster algebras. Namely, using the  skew-symmetric rule defining $\Omega'$, we define an infinite skew-symmetric matrix $\Lambda$ and we prove (Theorem \ref{thm compatibility}) that it is compatible (in the quantum cluster algebra setting) with the exchange matrix $B$ of the infinite basic quiver $\Gamma$ defined in \cite{hl16o}, for every finite Dynkin type. This means that we provide a quantization of the cluster algebra categorified by the category $\mathcal{O}^{\mathfrak{b},+}$  of representations of the quantum affine Borel subalgebra $\mathcal{U}_q(\hat{\mathfrak{b}})$ of the quantum affine algebra $\mathcal{U}_q(\hat{\mathfrak{g}})$, for underlying Lie algebra $\mathfrak{g}$ of every Dynkin type. Our result is a generalization of the result in \cite{b21}, which concerns simply-laced types.

To every Dynkin diagram of finite type we can associate an infinite periodic quiver $\Gamma$, as in \cite{hl16o}. We first consider the quiver $\tilde{\Gamma}$ having vertex set $V=I\times \ZZ$ and arrows given by
\begin{equation}
    (i,r)\to(j,s)\Leftrightarrow (C_{ij}\neq 0\ \text{and}\ s=r+r_iC_{ij}).
\end{equation}
The quiver $\tilde{\Gamma}$ has two connected components; we pick one of them and call it $\Gamma$ (see for example Figure \ref{fig:quiver}). 
\begin{figure}[]
    \centering
    $\begin{tikzcd}
        \vdots & \vdots && \\
        & {(2,5)} & \vdots \\
        && {(1,4)} \\
        & {(2,3)} && \vdots \\
        &&& {(1,2)} \\
        & {(2,1)} \\
        {(1,0)} \\
        & {(2,-1)} \\
        && {(1,2)} \\
        & {(2,-3)} \\
        &&& {(1,-4)} \\
        & {(2,-5)} && \vdots \\
        {(1,-6)} && {(1,-4)} \\
        \vdots & {(2,-7)} & \vdots \\
        & \vdots
        \arrow[no head, from=2-2, to=1-2]
        \arrow[from=2-2, to=5-4]
        \arrow[no head, from=3-3, to=2-3]
        \arrow[from=3-3, to=6-2]
        \arrow[from=4-2, to=2-2]
        \arrow[from=4-2, to=7-1]
        \arrow[no head, from=5-4, to=4-4]
        \arrow[from=5-4, to=8-2]
        \arrow[from=6-2, to=4-2]
        \arrow[from=6-2, to=9-3]
        \arrow[no head, from=7-1, to=1-1]
        \arrow[from=7-1, to=10-2]
        \arrow[from=8-2, to=6-2]
        \arrow[from=8-2, to=11-4]
        \arrow[from=9-3, to=3-3]
        \arrow[from=9-3, to=12-2]
        \arrow[from=10-2, to=8-2]
        \arrow[from=10-2, to=13-1]
        \arrow[from=11-4, to=5-4]
        \arrow[from=11-4, to=14-2]
        \arrow[from=12-2, to=10-2]
        \arrow[from=12-4, to=11-4]
        \arrow[from=13-1, to=7-1]
        \arrow[from=13-3, to=9-3]
        \arrow[from=14-1, to=13-1]
        \arrow[from=14-2, to=12-2]
        \arrow[from=14-3, to=13-3]
        \arrow[from=15-2, to=14-2]
    \end{tikzcd}$
    \caption{The quiver $\Gamma$ in type $G_2$.}
    \label{fig:quiver}
\end{figure}
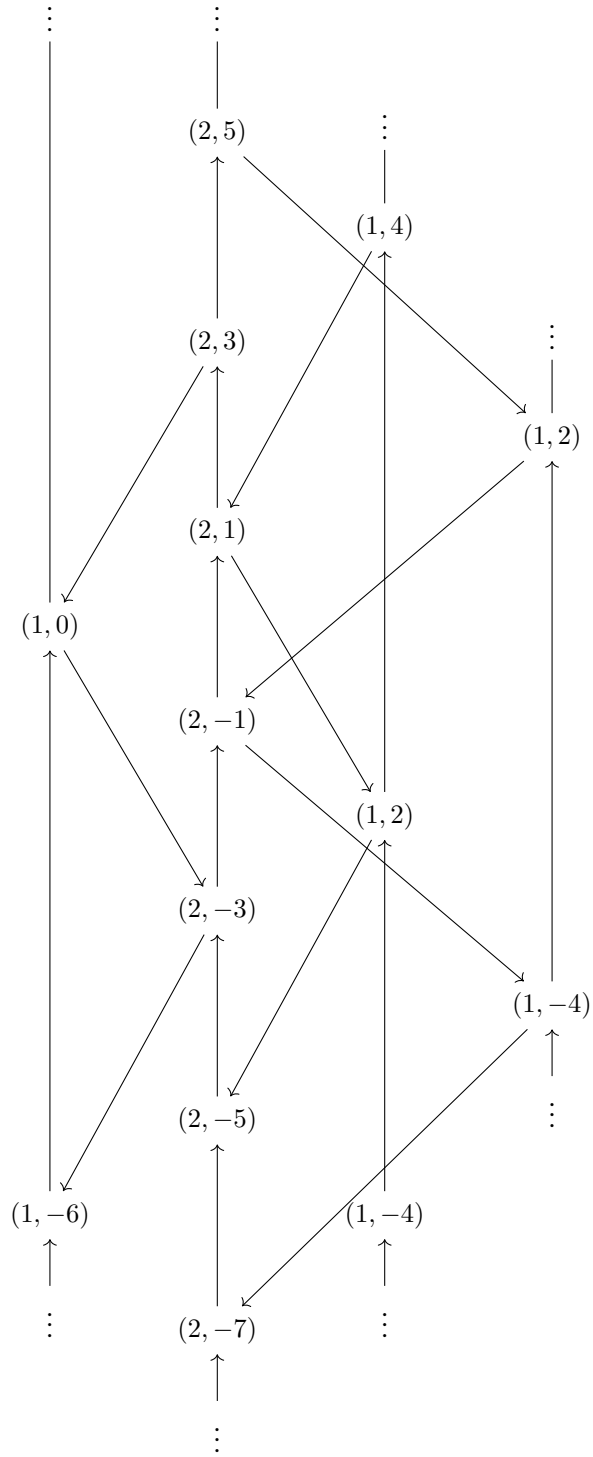
Note that between two vertices of $\Gamma$ there is at most one arrow. We denote by $V'$ the vertex set of $\Gamma$. Now let us call $B=(b_{(i,r)(j,s)})_{(i,r)(j,s)\in V'}$ the exchange matrix of $\Gamma$, in the sense of cluster algebra. This means that the entries of $B$ are $b_{(i,r)(j,s)}=1$(resp. $-1$) if there is an arrow in $\Gamma$ from $(i,r)$ to $(j,s)$ (resp. from $(j,s)$ to $(i,r)$). Let $\mathcal{A}'$ be the cluster algebra having an initial seed with quiver $\Gamma$ (note that this algebra is different from the cluster algebra $\mathcal{A}$ discussed in Section \ref{section QQ}). We now construct a quantization of this cluster algebra. Namely, we define a skew-symmetric integer $V'\times V'$-matrix $\Lambda$, that verifies the following compatibility condition: the matrix product $B\Lambda$ is a diagonal matrix and its diagonal entries are all strictly positive integers. In this case, we say that $(B,\Lambda)$ is a compatible pair. Note that $B\Lambda$ is well defined since $B$ is locally finite, i.e. every row and column of $B$ has only finitely-many non-zero entries.
\begin{definition} 
    Define $\Lambda=(\lambda_{(i,r)(j,s)})_{(i,r),(j,s)\in V'}$ to be the matrix defined by
    \[\lambda_{(i,r)(j,s)}:=\Omega'(\Psi_{i,q^r},\Psi_{j,q^s}).\] 
\end{definition}

 By Proposition \ref{F antisimmetrica}, we have that $\Lambda$ is skew-symmetric.

In \cite[Proposition 6.2.4]{b21} the author proved that for simply-laced type Dynkin diagrams, the pair $(B,\Lambda)$ is compatible (in that paper, the bicharacter is defined only for simply-laced types). We extend this result to non simply-laced types, with the general definition of $\Omega'$ that we have given in Definition \ref{def F} (this is an instance that our definition of $\Omega'$ is the "good" one). 

\begin{theorem}\label{thm compatibility}
    The pair $(B,\Lambda)$ is compatible. Precisely, we have
    \[B\Lambda=2\mathrm{Id}_{V'}.\]
\end{theorem}
\begin{proof}
    For every $(i,r),(j,s)\in V'$, we have
    \begin{align}
    \label{blambda}
        (B\Lambda)_{(i,r)(j,s)}&=\sum_{(k,t)\in V'}b_{(i,r)(k,t)}\lambda_{(k,t)(j,s)}=\\
        &=\lambda_{(i,r+2r_i)(j,s)}-\lambda_{(i,r-2r_i)(j,s)}+\sum_{\{k\in I\mid C_{ik}=-1\}}(\lambda_{(k,r-r_i)(j,s)}-\lambda_{(k,r+r_i)(j,s)})+\nonumber \\
         &\quad+\sum_{\{k\in I\mid C_{ik}=-2\}}(\lambda_{(k,r-2r_i)(j,s)}-\lambda_{(k,r+2r_i)(j,s)})+\sum_{\{k\in I\mid C_{ik}=-3\}}(\lambda_{(k,r-3r_i)(j,s)}-\lambda_{(k,r+3r_i)(j,s)}),\nonumber
    \end{align}
    where we have used the definition of $\Gamma$, or equivalently of $B$. We divide the proof according to the possible values of $r_i$ and $r_j$, but first we treat separately type $G_2$, which serves as an illustration of the general behavior of the sum in \eqref{blambda}.
    
    \textcolor{purple}{Compatibility in type $G_2$.}\\
    In this case the Cartan matrix is
    \[C(G_2)=\begin{pmatrix}
        2 & -1\\
        -3 & 2
    \end{pmatrix},\]
    hence $r_1=3$ and $r_2=1$. The relations that we obtain from Lemma \ref{relation quantum cartan} are the following:
    \begin{align}
        \tilde{c}_{12}(u+3)+\tilde{c}_{12}(u-3)&=\tilde{c}_{22}(u), \forall u\in\ZZ,\label{g1}\\
        \tilde{c}_{11}(u+3)+\tilde{c}_{11}(u-3)&=\tilde{c}_{21}(u),\forall u\in\ZZ\setminus\{0\},\label{g2}\\
        \tilde{c}_{21}(u+1)+\tilde{c}_{21}(u-1)&=\tilde{c}_{11}(u-2)+\tilde{c}_{11}(u)+\tilde{c}_{11}(u+2), \forall u\in\ZZ,\label{g3}\\
        \tilde{c}_{22}(u+1)+\tilde{c}_{22}(u-1)&=\tilde{c}_{12}(u-2)+\tilde{c}_{12}(u)+\tilde{c}_{12}(u+2), \forall u \in\ZZ\setminus\{0\}\label{g4}.
    \end{align}
     
    We start by computing
    \begin{equation}\label{g2 somma}
        (B\Lambda)_{(1,r)(2,s)}=\lambda_{(1,r+6)(2,s)}-\lambda_{(1,r-6)(2,s)}+\lambda_{(2,r-3)(2,s)}-\lambda_{(2,r+3)(2,s)}.
    \end{equation}
    Let us denote
    \[m=s-r,\] as we did in the previous sections.
    We observe the sum in \eqref{g2 somma} is telescopic and the remaining terms are
    \begin{align}\label{g2 somma 2}
        (B\Lambda)_{(1,r)(2,s)}=&\Big(\tilde{c}_{12}(m+5)+\tilde{c}_{12}(m+3)+\tilde{c}_{12}(m+1)+\tilde{c}_{12}(m-1)+\tilde{c}_{12}(m-3)+\tilde{c}(m-5)-\\&-\tilde{c}_{22}(m+2)-\tilde{c}_{22}(m)-\tilde{c}_{22}(m-2)\Big)+\Big(\tilde{c}_{12}(-m+5)+\tilde{c}_{12}(-m+3)+\tilde{c}_{12}(-m+1)+\nonumber\\&+\tilde{c}_{12}(-m-1)+\tilde{c}_{12}(-m-3)
        +\tilde{c}(-m-5)-\tilde{c}_{22}(-m+2)-\tilde{c}_{22}(-m)-\tilde{c}_{22}(-m-2)\Big)\nonumber.
    \end{align}

    We apply now \eqref{g1} to \eqref{g2 somma 2} (with $u=\pm m+2$, $u=\pm m$ and $u=\pm m-2$) and  we conclude that $(B\Lambda)_{(1,r)(2,s)}=0$ . Let us denote $p=r-s$ and consider now the entry
    \begin{align*}
        &(B\Lambda)_{(2,s)(1,r)}=\lambda_{(2,s+2)(1,r)}-\lambda_{(2,s-2)(1,r)}+\lambda_{(1,s-3)(1,r)}-\lambda_{(1,s+3)(1,r)}=\\
        &=\sum_{\ell\geq 1}\big(-\tilde{c}_{21}(p-2-6\ell+3)+\tilde{c}_{21}(p+2-6\ell+3)\big)+\sum_{\ell\geq 1}(-\tilde{c}_{11}(p-3-6\ell+3)+\tilde{c}_{11}(p+3-6\ell+3))+\\
        &+\sum_{\ell\geq 1}\big(-\tilde{c}_{21}(-p-2-6\ell+3)+\tilde{c}_{21}(-p+2-6\ell+3)\big)+\sum_{\ell\geq 1}\big(-\tilde{c}_{11}(-p-3-6\ell+3)+\tilde{c}_{11}(-p+3-6\ell+3)\big).
    \end{align*}
    We observe that if we rewrite \eqref{g3} with the shifts $u=u'+1$ and $u=u'-1$ and then subtract one from the other, we obtain:
    \begin{equation}\label{g5}
        \tilde{c}_{21}(u'+2)-\tilde{c}_{21}(u'-2)=\tilde{c}_{11}(u'+3)-\tilde{c}_{11}(u'-3), \forall u'\in\ZZ.
    \end{equation}
    Thus, for every $\ell\geq 1$, we can apply \eqref{g5} in the computation of $(B\Lambda)_{(2,s)(1,r)}$, and we obtain 
    \[(B\Lambda)_{(2,s)(1,r)}=0.\]
    We continue by computing:
    \begin{align*}
        (B\Lambda)_{(1,r)(1,s)}=&\lambda_{(1,r+6)(1,s)}-\lambda_{(1,r-6)(1,s)}+\lambda_{(2,r-3)(1,s)}-\lambda_{(2,r+3)(1,s)}=\\
        =&\big(\tilde{c}_{11}(m+3)+\tilde{c}_{11}(m-3)-\tilde{c}_{21}(m)\big)+\big(\tilde{c}_{11}(-m+3)+\tilde{c}_{11}(-m-3)-\tilde{c}_{21}(-m)\big).
    \end{align*}
    If $r\neq s$, i.e. $m\neq 0$, by \eqref{g2} we conclude that $(B\Lambda)_{(1,r)(1,s)}=0$. Otherwise, we have 
    \begin{align*}
        (B\Lambda)_{(1,r)(1,r)}=2\tilde{c}_{11}(3)=2,
    \end{align*}
    using the results from Proposition \ref{prop tecnica}. Finally we compute 
    \begin{align*}
        (B\Lambda)_{(2,r)(2,s)}=&\lambda_{(2,r+2)(2,s)}-\lambda_{(2,r-2)(2,s)}+\lambda_{(1,r-3)(2,s)}-\lambda_{(1,r+3)(2,s)}=\\
        =&\big(\tilde{c}_{22}(m+1)+\tilde{c}_{22}(m-1)-\tilde{c}_{12}(m+2)-\tilde{c}_{12}(m)-\tilde{c}_{12}(m-2)\big)+\\&+\big(\tilde{c}_{22}(-m+1)+\tilde{c}_{22}(-m-1)-\tilde{c}_{12}(-m+2)-\tilde{c}_{12}(-m)-\tilde{c}_{12}(-m-2)\big).
    \end{align*}
    For $s\neq r$, i.e. $m\neq0$ we conclude $(B\Lambda)_{(2,r)(2,s)}=0$ using \eqref{g4}. If $r=s$, i.e. $m=0$, then we have
    \[(B\Lambda)_{(2,r)(2,r)}=2\tilde{c}_{22}(1)=2,\] where again we have used Proposition \ref{prop tecnica}. Thus, we have proved that for underlying $\g$ of type $G_2$ we have
    \[B\Lambda=2\mathrm{Id}_{V'},\] so the couple $(B,\Lambda)$ is compatible.

\textcolor{purple}{Compatibility for all the other types.}\\
Since we are not considering type $G_2$, the sum in \eqref{blambda} simplifies to:
\begin{equation}\label{blambda no g2}
    \begin{split}
    (B\Lambda)_{(i,r)(j,s)}=&\lambda_{(i,r+2r_i)(j,s)}-\lambda_{(i,r-2r_i)(j,s)}+\sum_{\{k\in I\mid C_{ik}=-1\}}(\lambda_{(k,r-r_i)(j,s)}-\lambda_{(k,r+r_i)(j,s)})+ \\
        &+\sum_{\{k\in I\mid C_{ik}=-2\}}(\lambda_{(k,r-2r_i)(j,s)}-\lambda_{(k,r+2r_i)(j,s)}). 
\end{split}
\end{equation}
\textit{Notation}. Note that in \eqref{g2 somma 2} we have separated the summands appearing in the definition of each entry of $\Lambda$, so that we have first all terms of the form $\tilde{c}_{h,k}(m+...)$ and then all terms of the form $\tilde{c}_{h,k}(-m+...)$. We do the same in the general case and we refer to the second block of terms  (those of type $\tilde{c}_{h,k}(-m+...)$) as $S(-m)$. To make the text more readable, we omit the computations for $S(-m)$, as it can be treated using the same approach as the first part of the sum (just like in $G_2$).

As before, we keep denoting $m=s-r$. By definition of $\Lambda$ this is equivalent to
\begin{align*}
    (B\Lambda)_{(i,r)(j,s)}=&\sum_{\ell\geq1}\left(-\tilde{c}_{ij}(m-2r_i+r_j(-2\ell+1))+\tilde{c}_{ij}(m+2r_i+r_j(-2\ell+1))\right)+\\+&\sum_{\ell\geq 1}\left(\sum_{\{k\in I \mid C_{ik}=-1\}}-\tilde{c}_{kj}(m+r_i+r_j(-2\ell+1))+\tilde{c}_{kj}(m-r_i+r_j(-2\ell+1))\right)+\\+&\sum_{\ell\geq 1}\left(\sum_{\{k\in I\mid C_{ik}=-2\}}-\tilde{c}_{kj}(m+2r_i+r_j(-2\ell+1))+\tilde{c}_{kj}(m-2r_i+r_j(-2\ell+1))\right)+\\+&S(-m).
\end{align*}
We divide our discussion depending on the values of $r_i,r_j$. However, the computations have the same structure as in type $G_2$: we find that the sum above is telescopic and then apply Lemma \ref{relation quantum cartan}. 

\textcolor{purple}{Case $r_i=r_j=1$.}\\
In this case we have
\begin{align*}
    (B\Lambda)_{(i,r)(j,s)}=&\sum_{\ell\geq1}\left(-\tilde{c}_{ij}(m-1+-2\ell)+\tilde{c}_{ij}(m+3+-2\ell)\right)+\\&+\sum_{\ell\geq 1}\left(\sum_{\{k\in I\mid C_{ik}=-1\}}-\tilde{c}_{kj}(m+2-2\ell)+\tilde{c}_{kj}(m-2\ell)\right)+\\&+\sum_{\ell\geq 1}\left(\sum_{\{k\in I\mid C_{ik}=-2\}}-\tilde{c}_{kj}(m+3-2\ell)+\tilde{c}_{kj}(m-1+-2\ell)\right)+\\+& S(-m).
\end{align*}
All of these three sums (and those in $S(-m)$) are telescopic and the remaining terms are
\begin{align*}
    (B\Lambda)_{(i,r)(j,s)}=\tilde{c}_{ij}(m+1)+\tilde{c}_{ij}(m-1)-\sum_{\{k\in I\mid C_{ik}=-1\}}\tilde{c}_{kj}(m)-\sum_{\{k\in I\mid C_{ik}=-2\}}\tilde{c}_{kj}(m+1)+\tilde{c}_{kj}(m-1)+S(-m).
\end{align*}
When $(i,r)\neq (j,s)$, we can apply Lemma \ref{relation quantum cartan} and conclude that $(B\Lambda)_{(i,r)(j,s)}=0$. Otherwise, if $(i,r)=(j,s)$, applying Proposition \ref{prop tecnica}, we have two contributions of $\tilde{c}_{ii}(1)$ (one coming from $S(-m)$) and we conclude that
\[(B\Lambda)_{(i,r)(i,r)}=2\tilde{c}_{ii}(1)=2.\]

\textcolor{purple}{Case $r_i=1, r_j=2$.}\\
We are considering two distinct nodes in the Dynkin diagram and we are computing an off-diagonal entry of $B\Lambda$. Equation \eqref{blambda no g2} becomes:
\begin{align}\label{somma caso 2}
    (B\Lambda)_{(i,r)(j,s)}=&\sum_{\ell\geq1}\Big(-\tilde{c}_{ij}(m-2+2(-2\ell+1))+\tilde{c}_{ij}(m+2+2(-2\ell+1))+\\&+\sum_{\{k\in I\mid C_{ik}=-1\}}-\tilde{c}_{kj}(m+1+2(-2\ell+1))+\tilde{c}_{kj}(m-1+2(-2\ell+1))+\nonumber\\&+\sum_{\{k\in I\mid C_{ik}=-2\}}-\tilde{c}_{kj}(m+2+2(-2\ell+1))+\tilde{c}_{kj}(m-2+2(-2\ell+1))\Big)+\nonumber\\& +S(-m)\nonumber.
\end{align}
Let us now consider Lemma \ref{relation quantum cartan} for indices $i,j$. We have:
\begin{equation}\label{caso 2}
    \tilde{c}_{ij}(u+1)+\tilde{c}_{ij}(u-1)-\sum_{\{k\in I\mid C_{ik}=-1\}}\tilde{c}_{kj}(u)-\sum_{\{k\in I\mid C_{ik}=-2\}}\left(\tilde{c}_{kj}(u+1)+\tilde{c}_{kj}(u-1)\right)=0
\end{equation}
Rewriting \eqref{caso 2} shifting the arguments by $\pm 1$, and subtracting one equation from the other, we obtain
\begin{equation}\label{caso 2 2}
    \tilde{c}_{ij}(u+2)-\tilde{c}_{ij}(u-2)+\sum_{\{k\in I\mid C_{ik}=-1\}}\left(-\tilde{c}_{kj}(u+1)+\tilde{c}_{kj}(u-1)\right)-\sum_{\{k\in I\mid C_{ik}=-2\}}\left(-\tilde{c}_{kj}(u+2)+\tilde{c}_{kj}(u-2)\right)=0.
\end{equation}
For every $\ell$ in the sum \eqref{somma caso 2} we apply \eqref{caso 2 2}, with $u=m+2(-2\ell+1)$, and this let us conclude that the global sum is $0$. Thus, we have proved 
\[(B\Lambda)_{(i,r)(j,s)}=0\] for these values of $r_i$ and $r_j$.

\textcolor{purple}{Case $r_i=2, r_j=1$.}\\
Before doing any computation, observe that this hypothesis implies $C_{ji}=-2$, hence the subset of $k\in I$ such that $C_{ik}=-2$ is empty. This simplifies the sum that we are considering. We have:

\begin{align*}
    (B\Lambda)_{(i,r)(j,s)}=&\sum_{\ell\geq1}\left(-\tilde{c}_{ij}(m-3-2\ell)+\tilde{c}_{ij}(m+5-2\ell)\right)+\\+&\sum_{\ell\geq 1}\left(\sum_{\{k\in I\mid C_{ik}=-1\}}-\tilde{c}_{kj}(m+3-2\ell)+\tilde{c}_{kj}(m-1-2\ell)\right)+ S(-m).
\end{align*}
Again, these sums are telescopic and we are left with
\begin{align*}
    (B\Lambda)_{(i,r)(j,s)}&=\tilde{c}_{ij}(m+1)+\tilde{c}_{ij}(m-1)+\tilde{c}_{ij}(m+3)+\tilde{c}_{ij}(m-3)-\\&-\sum_{\{k\in I\mid C_{ik}=-1\}}\left(\tilde{c}_{kj}(m+1)+\tilde{c}_{kj}(m-1)\right)+ S(-m).
\end{align*}
On the other hand, in this setting Lemma \ref{relation quantum cartan} reads:
\begin{equation}\label{caso3}
    \tilde{c}_{ij}(u-2)+\tilde{c}_{ij}(u+2)-\sum_{\{k\in I\mid C_{ik}=-1\}}\tilde{c}_{kj}(u)=0.
\end{equation} Thus, plugging relation \eqref{caso3} in our sum, with the appropriate shifts $u=m\pm 1$, we obtain that $(B\Lambda)_{(i,r)(j,s)}=0$.

\textcolor{purple}{Case $r_i=r_j=2$.}\\
Note that $r_i=2$ implies that there is no $k\in I$ such that $C_{ik}=-2$. Substituting the values $r_i=r_j=2$, we get:
\begin{equation}\label{caso 4}
\begin{split}
    (B\Lambda)_{(i,r)(j,s)} =& \sum_{\ell \ge 1} \Big( 
        -\tilde{c}_{ij}(m-2-4\ell) + \tilde{c}_{ij}(m+6-4\ell)+ \sum_{\{k\mid C_{ik}=-1\}}\big(-\tilde{c}_{kj}(m+4-4\ell) + \tilde{c}_{kj}(m-4\ell)\big) 
    \Big)\\&+ S(-m).
\end{split}
\end{equation}
We observe that the sum is telescopic, leaving us with
\begin{equation}\label{caso 4 2}
    (B\Lambda)_{(i,r)(j,s)}=\tilde{c}_{ij}(m+2)+\tilde{c}_{ij}(m-2)+\sum_{\{k\in I\mid C_{ik}=-1\}}\tilde{c}_{kj}(m)+S(-m).
\end{equation}
As in the previous cases, we can apply Lemma \ref{relation quantum cartan} to conclude that if $(i,r)\neq(j,s)$, then $(B\Lambda)_{(i,r)(j,s)}=0$. Finally, if $(i,r)=(j,s)$, we get
\[(B\Lambda)_{(i,r)(i,r)}=2\tilde{c}_{ii}(2)=2.\]

We have shown that every diagonal entry of $B\Lambda$ is equal to $2$ and every off-diagonal entry is zero. This concludes the proof.
\end{proof}
Let us comment about this result. First we give a definition:
\begin{definition}
    We define $\mathcal{A}'_t$ to be the quantum cluster algebra with an initial seed given by the quiver $\Gamma$ and quantization matrix $\Lambda$.
\end{definition}
By construction, $\mathcal{A}'_t$ is a subring of $\mathcal{Y}'_t$, and in simply-laced types it coincides with the quantum cluster algebra constructed in \cite{b21}. In the same work, the author defined the quantum Grothendieck ring for $\mathcal{O}^{\mathfrak{b},+}_{\ZZ}$ as a completion of such quantum cluster algebra, and proved that it contains $K_t(\mathcal{C}_{\ZZ})$, the quantum Grothendieck ring for finite-dimensional modules defined in \cite{nak04} and \cite{vv}. Thus, our generalization of the definition of $\mathcal{A}'_t$ suggests to prove that also in non simply-laced setting we can recover this inclusion of rings. We plan to do this in a forthcoming paper. 
%\section{Perspectives}\label{perspectives}
%\input{perspectives}

\printbibliography[heading=bibintoc]

\begin{tabular}{ll}
FRANCESCA PAGANELLI &
DEPARTMENT S.B.A.I.\\ & SAPIENZA-UNIVERSITÀ DI ROMA\\
& VIA SCARPA 10, 00164, ROMA (ITALY) \\
& \&\\
 &UNIVERSITÉ PARIS CITÉ, SORBONNE UNIVERSITÉ\\
& CNRS, IMJ-PRG\\
& 8 PL. AURÉLIE NEMOURS, 75013, PARIS (FRANCE)\\
email: {\tt francesca.paganelli@uniroma1.it}\\[5mm]

\end{tabular}

\end{document}